\documentclass[12pt,twoside,reqno]{amsart}
\allowdisplaybreaks
\usepackage{bm}

\usepackage{amssymb,amsmath,mathrsfs,amstext,amsthm,amsfonts,amscd,xcolor, tikz-cd}

\usepackage{bbm}
\usepackage{dsfont}
\usepackage{tablefootnote}

\usepackage{mathrsfs}

\usepackage{amsthm, enumerate}

\usepackage[ansinew]{inputenc}
\usepackage{graphicx}
\usepackage{hyperref}

\usepackage{footnote} 
\makesavenoteenv{tabular}

\newcommand{\blob}{\rule[.2ex]{.8ex}{.8ex}}

\newcommand{\R}{\mathbb{R}}
\newcommand{\C}{\mathbb{C}}
\newcommand{\N}{\mathbb{N}}
\newcommand{\Z}{\mathbb{Z}}

\newcommand{\HP}{\mathbb{H}}

\newcommand{\SL}{{\rm SL}}

\newcommand{\GL}{\rm GL}

\newcommand{\Mat}{{\rm Mat}}

\newcommand{\tr}{\mbox{tr}}

\newcommand{\Fscr}{\mathscr{F}}

\newcommand{\Ecal}{\mathcal{E}}	
\newcommand{\conorm}{\mathrm{m}}

\newcommand{\Leb}{{\rm Leb}}

\newcommand{\sabs}[1]{\left| #1 \right|} 
\newcommand{\abs}[1]{\bigl| #1 \bigr|} 
\newcommand{\norm}[1]{\lVert#1\rVert} 
\newcommand{\normtwo}[1]{
	{\left\vert\kern-0.25ex\left\vert\kern-0.25ex\left\vert #1
		\right\vert\kern-0.25ex\right\vert\kern-0.25ex\right\vert} }

\newcommand{\avg}[1]{\left< #1 \right>} 

\newcommand{\la}{\lambda}

\newcommand{\om}{\omega}

\newcommand{\Proj}{\mathbb{P}(\R^ 2)}

\newcommand{\Pp}{\mathbb{P}}
\newcommand{\Cp}{\mathbb{CP}}

\newcommand{\Bscr}{\mathscr{B}}

\newcommand{\uA}{\underline{A}}
\newcommand{\Ascr}{\mathscr{A}}
\newcommand{\Mscr}{\mathcal{M}}

\newcommand{\Kscr}{\mathscr{K}}
\newcommand{\Rscr}{\mathscr{R}}
\newcommand{\Wscr}{\mathscr{W}}
\newcommand{\Ainv}{\Ascr_{\mathrm{inv}}}
\newcommand{\Asing}{\Ascr_{\mathrm{sing}}}

\newcommand{\Matdm}{{\rm Mat}^+_2(\R)}

\newcommand{\Hscr}{\mathscr{H}}
\newcommand{\Gscr}{\mathscr{G}}
\newcommand{\tuA}{\underline{\tilde A}}
\newcommand{\Kinv}{K_{{\rm inv}}}
\newcommand{\Amap}{\hat A}

\newcommand\restr[2]{{
		\left.\kern-\nulldelimiterspace 
		#1 
		\vphantom{\big|} 
		\right|_{#2} 
}}

\theoremstyle{plain}
\newtheorem{theorem}{Theorem}[section]
\newtheorem{proposition}{Proposition}[section]
\newtheorem{corollary}{Corollary}[section]
\newtheorem{lemma}{Lemma}[section]

\theoremstyle{definition}
\newtheorem{definition}{Definition}[section]

\theoremstyle{definition}
\newtheorem{remark}{Remark}[section]
\newtheorem{example}[theorem]{Example}
\numberwithin{equation}{section}

\newcommand{\rank}{\mathrm{rank}}
\newcommand{\Ker}{\mathrm{Ker}}
\newcommand{\Range}{\mathrm{Range}}

\title[Random two dimensional cocycles]{Random 2D linear  cocycles I: \\dichotomic  behavior}

\date{}

\begin{document}

\author[P. Duarte]{Pedro Duarte}
\address{Departamento de Matem\'atica and CEMS.UL\\
Faculdade de Ci\^encias\\
Universidade de Lisboa\\
Portugal
}
\email{pmduarte@fc.ul.pt}

\author[M. Dur\~aes]{Marcelo Dur\~aes}
\address{Departamento de Matem\'atica, Pontif\'icia Universidade Cat\'olica do Rio de Janeiro (PUC-Rio), Brazil}
\email{accp95@gmail.com}

\author[T. Graxinha]{Tom\'e  Graxinha}
\address{Departamento de Matem\'atica and CEMS.UL\\
	Faculdade de Ci\^encias\\
	Universidade de Lisboa\\
	Portugal}
\email{tfgraxinha@fc.ul.pt}

\author[S. Klein]{Silvius Klein}
\address{Departamento de Matem\'atica, Pontif\'icia Universidade Cat\'olica do Rio de Janeiro (PUC-Rio), Brazil}
\email{silviusk@puc-rio.br}

\begin{abstract}
In this paper  we establish a Bochi-Ma\~n\'e type dichotomy in the space of two dimensional, nonnegative determinant matrix valued, locally constant linear cocycles over a  Bernoulli or Markov shift. Moreover, we prove that Lebesgue almost every such cocycle has finite first Lyapunov exponent, which then implies a break in the regularity of the Lyapunov exponent, from analyticity to discontinuity. 
\end{abstract}

\maketitle


\section{Introduction and statements}\label{aby}

Let $(X,\mu,f)$ be a measure preserving dynamical system, where the state space $X$ is a compact metric space, the transformation $f \colon X \to X$ is continuous and $\mu$ is an $f$-invariant, ergodic Borel probability measure on $X$. Given a bounded, measurable function $A \colon X\to \Mat_m (\R)$, 
we call a linear cocycle over the base dynamics $(X, \mu, f)$ with fiber dynamics driven by $A$
the skew-product map $F \colon X\times\R^m \to X\times\R^m$ defined by
$$F(x, v)=(f (x), \, A(f x) v) \, .$$
The iterates of this new dynamical system are
$F^n(x, v)=(f^n (x), \, A^n (x) v)$, where for all $n\in\N$, 
$$A^n (x):= A(f^{n} x)  \cdots A(f^2 x) \,  A(f x) \, .$$

The first Lyapunov exponent of a linear cocycle $F$ measures the asymptotic exponential growth of its fiber iterates. It is defined via the Furstenberg-Kesten theorem (the analogue of the additive ergodic theorem for multiplicative systems) as the $\mu$-a.e. limit
$$ L_1 (F) = L_1 (A) := \lim_{n\to\infty} \, \frac{1}{n} \, \log \norm{ A^n (x) } \, $$
provided  $A$ satisfies the integrability condition $\int_X \log^+ \norm{A}  d \mu < \infty$.

Moreover, if the norm (that is, the first singular value) of the fiber iterates is replaced by the second singular value, the corresponding $\mu$-a.e. limit above still exists, it is called the second Lyapunov exponent and it is denoted by $L_2 (F) = L_2 (A)$.

An important problem in ergodic theory---with deep applications elsewhere, e.g. in mathematical physics---concerns the regularity of the Lyapunov exponent as a function of the input data (e.g. the fiber map $A$, or the measure $\mu$ or even the base transformation $f$).\footnote{In this paper we actually fix the  base dynamics and identify the linear cocycle with its fiber map.} By regularity we mean its continuity properties (or lack thereof), including its modulus of continuity, its smoothness or analyticity in an appropriate setting etc. 
It turns out that this is a difficult problem in any context, its resolution strongly depending  on the topology of the space of cocycles, on the regularity of the fiber map and on the type of base dynamics considered. 

\smallskip

Indeed, a version of the Bochi-Ma\~n\'e dichotomy theorem in the context of linear cocycles states that given any measure preserving dynamical system $(X, \mu, f)$ where $f$ is a an aperiodic (meaning its set of periodic points has measure zero) homeomorphism, for any fiber map $A \colon X \to \SL_2(\R)$,  either the cocycle associated with $A$ is hyperbolic over the support of $\mu$ or else $A$ is approximated in the $C^0$ topology  by fiber maps with zero Lyapunov exponent. 

By Ruelle's theorem, hyperbolic cocycles are points of analyticity of the Lyapunov exponent; moreover, since the Lyapunov exponent is always upper semicontinuous, it must be continuous at any $\SL_2(\R)$-valued cocycle with zero Lyapunov exponent. Therefore, a fiber map $A \colon X \to \SL_2(\R)$ is a continuity point for the first Lyapunov exponent  in $C^0(X, \SL_2(\R))$ if and only if the corresponding linear cocycle is either hyperbolic over the support of $\mu$ (in which case the Lyapunov exponent is in fact analytic) or $L_1 (A)=0$. 

\smallskip

In other words, if $L_1 (A) > 0$, the first Lyapunov exponent is either analytic or discontinuous at $A$ in the $C^0$ topology of $ \SL_2(\R)$ cocycles.
This behavior may change dramatically when the fiber map is highly regular and it varies in a space endowed with an appropriate, strong topology.\footnote{However, we point out that currently the only reasonably well understood systems are linear cocycles over base dynamics given by: Bernoulli or Markov shifts, or more generally, uniformly hyperbolic maps as well as toral translations and  certain combinations thereof (mixed systems).} 

\smallskip

Indeed, consider a locally constant linear cocycle over a Bernoulli shift. More precisely, let $\Ascr:=\{1,\ldots, k\}$ be a finite alphabet \footnote{In general $\Ascr$ can be a compact metric space, or a non-compact, countable set of symbols, but in this paper we focus on the finite state space setting.} and let $p = (p_1, \ldots, p_k)$ be a probability vector with $p_i > 0$ for all $i$. Denote by $X:=\Ascr^\Z$ the space of bi-infinite sequences $\om = \{\om_n\}_{n\in\Z}$ on this alphabet, which we endow with  the product measure $\mu = p^\Z$. Let $\sigma \colon X \to X$ be the corresponding forward shift $\sigma \om = \{\om_{n+1}\}_{n\in\Z}$. Then $(X, \mu, \sigma)$ is a measure preserving, ergodic dynamical system called a Bernoulli shift (in finite symbols). A $k$-tuple $\underline{A} = (A_1, \ldots, A_k) \in \Mat_2 (\R)^k$ determines the locally constant fiber map $A \colon X \to \Mat_2 (\R)$, $A (\om) = A_{\om_0}$, which in turns determines a linear cocycle over the Bernoulli shift, referred to as a (random) Bernoulli cocycle. We identify this cocycle with the fiber map $A$ and with the tuple $\underline{A}$ and denote by $L_1 (\underline{A}) = L_1 (\underline{A}, p)$ its first Lyapunov exponent.

\smallskip

More generally, we also consider linear cocycles over a Markov shift. That is, let $P\in \Mat_k (\R)$ be a  (left) stochastic matrix, i.e. $P=(p_{i j})_{1\leq i,j\leq k}$ with $p_{ij}\geq 0$
and $\sum_{i=1}^k p_{i j} = 1$ for all $1\le j \le k$.
Given a $P$-stationary probability vector $q$, i.e.,
$q=P\, q$, the pair $(P, q)$ determines a probability measure $\mu$ on $X$  for which the process
$\xi_n \colon X\to \Ascr$, $\xi_n(\omega):=\omega_n$ is
a stationary Markov chain in $\Ascr$ with probability transition 
matrix $P$ and initial distribution law $q$. 
Then 
$(X, \sigma, \mu)$ is a measure preserving dynamical system called a Markov shift.

The stochastic matrix $P$ determines a directed weighted graph on the vertex set $\Ascr$ with an edge $j\mapsto i$, from $j$ to $i$, whenever $p_{i j}>0$.
Sequences in $X$ describing paths in this graph are called $P$-admissible.
In this setting we employ the notation $X$ to refer to the  set  of all $P$-admissible sequences, in other words, to  the support of the Markov measure $\mu$.
Given a finite admissible word $(i_0, i_1,\ldots, i_n)\in \Ascr^{n+1}$ and $k\in\Z$, the set 
$$[k; i_0, i_1, \ldots, i_n] := \left\{ \om \in X  \colon \om_{k+l} = i_l \quad  \text{for all } 0 \le l \le n \right\} $$
is called a cylinder of length $n+1$ in $X$.
Its (Markov) measure is then
$$\mu \left( [k; i_0, i_1, \ldots, i_n] \right) = q_{i_0} \, p_{i_i, i_0} \, \cdots \, p_{i_n, i_{n-1}} \, .$$

We will assume that the matrix $P$ is primitive, i.e. there exists $n\ge1$ such that  $p^n_{i j}>0$ for all entries of the power matrix $P^n$. Then
$\lim_{n\to \infty} p^n_{i j} = q_i > 0$
for all $1\le i, j \le k$ and the corresponding Markov shift $(X, \mu, \sigma)$ is ergodic and mixing.
As before, a $k$-tuple $\underline{A} = (A_1, \ldots, A_k) \in \Mat_2 (\R)^k$ determines a locally constant linear cocycle over this base dynamics, which we refer to as a (random) Markov cocycle. Its first Lyapunov exponent is denoted by $L_1 (\underline{A}) = L_1 (\underline{A}, P, q)$.

\medskip

When we restrict to {\em invertible} random linear cocycles, that is, when $\underline{A} = (A_1, \ldots, A_k) \in \GL_2 (\R)^k$, the Lyapunov exponent is always continuous; moreover, it has a good modulus of continuity (at least in the Bernoulli case and when case  $L_1 > L_2$). This has been the subject of intense research throughout the years, starting with the celebrated work of Furstenberg and Kifer~\cite{FK83}.  The same holds (trivially) when all matrices $A_1, \ldots, A_k$ defining the cocycle $\uA$ are singular (noninvertible).  

\smallskip

However, as it turns out, the behavior of the first Lyapunov exponent at random cocycles with {\em both singular and invertible} components  is strikingly different, showing a dichotomy in the spirit of Bochi-Ma\~n\'e's.

In this paper we consider random cocycles with values in $\Matdm$, the multiplicative semigroup of all $g\in\Mat_2(\R)$ with $\det (g) \ge 0$.
The proof  is based in part upon an extension to $\Mat_2(\R)$-valued  cocycles of a result of Avila, Bochi, Yoccoz~\cite{ABY10} characterizing the class of projectively uniformly hyperbolic (PUH) $\SL_2 (\R)$-valued cocycles.  
 A first step consists in extending this theory to $\SL_2'(\R)$-valued cocycles, where $\SL_2'(\R)$ is the non connected group 
$$\SL_2'(\R):=\{g\in \GL_2(\R)\colon \det(g)=\pm 1\, \} .$$

This extension is relatively simple in the (local) characterization of projective uniform hyperbolicity but surprisingly subtle in the (global) characterization of the boundary of the class of PUH cocycles.
For this reason the general case  of $\Mat_2(\R)$-valued cocycles  will be addressed in a future project, building upon the results and methods of the current work.

The main result of this paper is then the following.

\begin{theorem}
	\label{thm: BM type dichotomy}
Given $\uA\in \Matdm^k$ with at least one singular and one invertible component, the following dichotomy holds:
either $\uA$ is (projectively) uniformly hyperbolic or else it can be approximated by a sequence of cocycles $\uA_n\in \Mat_2^+(\R)^k$ with $L_1(\uA_n)=-\infty$ for all $n \in \N$.
\end{theorem}

Using Ruelle's theorem we then conclude the following.

\begin{corollary}
	\label{coro: BM type dichotomy}
Given $\uA\in \Matdm^k$ with at least one singular and one invertible  component, if $L_1(\uA)>-\infty$ then the following dichotomy holds: either the Lyapunov exponent $L_1$ is analytic at $\uA$ or else $L_1$ is discontinuous at $\uA$.
\end{corollary}

Moreover, an argument involving subharmonic functions  (see Corollary~\ref{cor2ManeBochi}) shows that
in fact $L_1(\uA)> - \infty$ holds for Lebesgue almost every $\uA\in \Matdm^k$. Furthermore, using a topological argument, the set of continuity points of $L_1$ is a Baire residual subset of $\Matdm^k$.

Putting together other recent results on random (Bernoulli or Markov) two dimensional cocycles in finite symbols, we obtain the following almost complete picture on the regularity of
their first Lyapunov exponent $L_1$.

\blob \; $L_1$ is a continuous function on $\Mat_2(\R)^k$ at any invertible cocycle $\uA$; moreover, its regularity  varies from log-H\"older continuous\footnote{Given a metric space $(M, d)$, a function $\phi\colon M \to \R$ is said to be log-H\"older continuous if $\abs{\phi (x) - \phi (y)} \le C \left( \log \frac{1}{d (x, y)} \right)^{-1}$ for some $C<\infty$ and all $x, y \in M$.} to analytic (see~\cite{Rue79a},~\cite{BV},~\cite{Viana-M},~\cite{Tall-Viana},~\cite{DKS19},~\cite{DK-CBM},~\cite{DK-Holder}).

\smallskip 

\blob\, In the algebraic variety
$$ \Rscr_1=\big\{ \uA\in \Mat_2 (\R)^k \colon  \rank(A_i)=1 \, \,  \forall \, 1\leq i\leq k \big\} $$
the map $L_1$ is always continuous. Moreover, every $\uA \in \Rscr_1$ is a continuity point of the Lyapunov exponent $L_1$ in $\Mat_2(\R)^k$. 

It turns out (see Theorem~\ref{ManeBochiRank1}) that a cocycle $\uA\in \Rscr_1$ is projectively uniformly hyperbolic if and only if $L_1(\uA)>-\infty$, which is equivalent to the absence of null words (i.e. vanishing finite products of components of $\uA$). This latter condition holds for almost every such cocycle. Therefore $L_1$ is analytic at Lebesgue almost every $\uA\in \Rscr_1$.

\smallskip 

\blob\, Given a cocycle $\uA\in \Matdm^k$ with at least one inver\-ti\-ble and one singular component,  either $L_1$ is analytic or it is discontinuous at $\uA$.

\smallskip 

\blob\, As mentioned above, the extension to the remaining case, when the cocycle $\uA\in \Mat_2(\R)^k$ admits a singular component and an invertible with negative determinant will be considered in a future work.

\medskip

We also prove the following (see Section~\ref{puh} for the definitions of relevant concepts).

\begin{theorem}
A random Bernoulli or Markov linear cocycle $\uA\in \Mat_2(\R)^k$ is projectively uniformly hyperbolic if and only if
it admits an invariant multi-cone.
\end{theorem}

Finally, we provide a formal model (see example~\ref{example CS}) of a one-parameter family of Bernoulli linear cocycles $\uA_t \in  \Mat_2^+ (\R)^2$ associated to the random Schr\"odinger operator whose potential takes two values, a finite one $a \in \R$ and $\infty$, as in Craig and Simon~\cite[Example 3]{CraigSimon}. For each of these cocycles, one component is invertible and the other one is singular. The map $t\mapsto L_1 (\uA_t)$ is almost everywhere discontinuous on the spectrum of this operator, which correlates well with the fact (proven by Craig and Simon) that a related physical quantity, the integrated density of states (IDS) of the operator is discontinuous.   
 
\medskip
The rest of the paper is organized as follows. 
In Section~\ref{aby} we extend two of the main theorems in~\cite{ABY10} from $\SL_2(\R)$-valued cocycles to $\GL_2(\R)$-valued cocycles.
In Section~\ref{puh} we formally introduce and characterize the concept of  PUH cocycle, extending the first of the previous results from $\GL_2(\R)$ to  $\Mat_2(\R)$-valued cocycles.  In Section~\ref{dichotomy} we establish, now in the context of $\Matdm$-valued cocycles,  the main result of this paper, Theorem~\ref{thm: BM type dichotomy}, and its consequences regarding the dichotomic behavior of random cocycles with both invertible and singular components.  Finally, in Section~\ref{examples} we provide several examples of classes of cocycles fitting the setting of this paper, including some that are relevant in mathematical physics.

\section{Avila-Bochi-Yoccoz Theory}\label{aby}
	
A. Avila, J. Bochi and J.C. Yoccoz studied 
 the class of uniformly hyperbolic $\SL_2(\R)$ cocycles over sub-shifts of finite type,
 characterizing its boundary within the space of all such cocycles. In this section we extend two results of their theory  to $\GL_2(\R)$ cocycles.

Throughout the section we fix a topologically mixing sub-shift of finite type $\sigma:X\to X$  over the alphabet $\Ascr=\{1, \ldots, k\}$, where $X$ denotes the space of admissible sequences.
This could be determined by the choice of a primitive stochastic $k$ by $k$ matrix.
As before, any tuple $\uA\in \GL_2(\R)^\Ascr$ determines the linear cocycle
$F_{\uA}:X\times\R^2\to X\times\R^2$ by
$F_{\uA}(\omega, v):=(\sigma\omega, A_{\omega_1}\, v)$.

Below $G$ refers to any of the following three groups. 
\begin{align*}
	\SL_2(\R) &:=\{ g\in \Mat_2(\R) \colon \det(g)=1\, \}\\
	\SL_2'(\R) &:=\{ g\in \Mat_2(\R) \colon \det(g)=\pm 1\, \}\\
	\GL_2(\R) &:=\{ g\in \Mat_2(\R) \colon \det(g)\neq 0 , \} 
\end{align*}

A cocycle $F_{\uA}:X\times \R^2 \to X\times \R^2$
determined by a tuple
$\uA\in \GL_2(\R)^\Ascr$ is said to be
\textit{projectively uniformly hyperbolic} if 
it has dominated splitting. A precise definition in the more general framework of $\Mat_2(\R)$-valued cocycles is given in Definition~\ref{def puh mat2}.
	
	The following two theorems hold for any of these groups. See the definition of \textit{multi-cone} in the Bernoulli and Markov frameworks after theorems 2.2 and 2.3 in~\cite{ABY10}, otherwise see definitions~\ref{multi-cone: Bernoulli} and~\ref{multi-cone: Markov} below.

\begin{theorem}
	\label{ABY1}
	Over the given sub-shift, a cocycle $\uA\in G^\Ascr$ is projectively uniformly hyperbolic \, iff\, 
	there exists an $\uA$-invariant multi-cone.
\end{theorem}

The case $G=\SL_2(\R)$ was established in~\cite[theorems 2.2 and 2.3]{ABY10}.

\bigskip

Next denote by  $\mathscr{H}(G)$  the space of projectively uniformly hyperbolic cocycles
$\uA\in G^\Ascr$ cocycles over the given sub-shift of finite type.
The set $\mathscr{H}(G)$ is open in $G^\Ascr$  and has countably many connected components if  $\Ascr$ has more than one element.

\medskip

We say that a matrix $g\in \GL_2(\R)$ is
\begin{itemize}
	\item \textit{projectively hyperbolic} if it possesses eigenvalues with distinct absolute values. This means that the projective action of $g$ has both an attractive and a repelling fixed points, respectively denoted by $s(g)$ and $u(g)$.
	\item  \textit{projectively elliptic} if  it possesses non real complex conjugate  eigenvalues. This means that the projective action of $g$ has no fixed points, and hence must be a $\theta$-rotation by some angle $0<\theta<\pi$. In this case  $\det(g)>0$.
	\item  \textit{projectively parabolic} if it is non diagonalizable but has a double real  eigenvalue. This means that the projective action of $g$ has a unique fixed point. In this case  $\det(g)>0$.
	\item a \textit{reflection} if it has eigenvalues $1$ and $-1$.
	This means that the projective action of $g$ is an involution with two fixed points.
	\item a \textit{dilation} if $g=c\, I$ for some $c\neq 0$. This means that the projective action of $g$ fixes all points.
\end{itemize}

The previous classification exhausts all possibilities for $g \in \GL_2(\R)$.

\begin{theorem}
	\label{ABY2}
 If $\uA\in \partial \,\Hscr(G)$  then one of the following possibilities hold:
 \begin{enumerate}
 	\item There exists a periodic point $\omega\in X$, of period $k$, such that $A^k(\omega)$  is either  projectively parabolic or 
 	a dilation.
 	\item There exist periodic points $\omega, \omega'\in X$, of respective periods $k$ and $\ell$, such that the matrices $A^k(\omega)$ and $A^\ell(\omega')$  are projectively hyperbolic 
 	and there exist an integer $m\geq 0$ and a point $z\in W^u_{\rm loc}(\omega) \cap \sigma^{-n} W^u_{\rm loc}(\omega')$  such
 	that
 	$$ A^n(z)\, u(A^k(\omega))=  s(A^\ell(\omega')) .$$
 \end{enumerate}
Furthermore, for each connected component  of $\Hscr(G)$, one can give uniform bounds to the numbers
$k$, $\ell$, $n$  that may appear in the alternatives above.
\end{theorem}

When $G=\SL_2(\R)$ this theorem  was established in~\cite[Theorem 2.3]{ABY10}.
We will refer to these two theorems as the ABY theorems 1 and 2.  

\bigskip

\noindent 
{\bf Reduction from $\GL_2(\R)$  to $\SL_2'(\R)$}
 
\noindent
Assume for now that the ABY theorems hold for  $G=\SL_2'(\R)$.
Define for each $g\in \GL_2(\R)$, $g^\ast:=(|\det(g)|)^{-1/2}\, g\in \SL_2'(\R)$.
Notice that the map $\GL_2(\R)\ni g\mapsto g^\ast\in \SL_2'(\R)$ is a group homomorphism. The proofs of the following key but obvious remarks will be omitted.

\begin{remark}
	Given $g\in \GL_2(\R)$, 
	\begin{itemize}
		\item $g$ and $g^\ast$ induce the same projective action
		$\hat g:\Proj\to \Proj$;
		\item $g$  is projectively hyperbolic $\Leftrightarrow$ $g^\ast$ is hyperbolic;
		\item $g$ is projectively elliptic  
		$\Leftrightarrow$ $g^\ast$ is elliptic;
		\item $g$ is projectively parabolic  
		$\Leftrightarrow$ $g^\ast$ is parabolic;
		\item $g$ is a dilation  $\Leftrightarrow$ $g^\ast=\pm\, I$.
	\end{itemize}
\end{remark}

\begin{remark}
	Given a cocycle $\uA\in \GL_2(\R)^\Ascr$, 
	\begin{itemize}
		\item $\uA$  is uniformly projectively hyperbolic $\Leftrightarrow$ $\uA^\ast$ is hyperbolic;
		\item $\uA$ and $\uA^\ast$ have the same invariant cones.
	\end{itemize}
\end{remark}

These remarks imply that  ABY theorems also hold for $G=\GL_2(\R)$
by reduction to $G=\SL_2'(\R)$.
We are now left to prove that the ABY theorems hold for $G=\SL_2'(\R)$.

\bigskip

\noindent 
{\bf Reduction from $\SL_2'(\R)$  to $\SL_2(\R)$}

\noindent
 A cocycle $\uA\in \SL_2'(\R)^\Ascr$ can be regarded as a bundle map
$$F_{\uA}:X\times \R^2\to X\times \R^2,\; F_A(\omega, v)= (\sigma\omega, A_{\omega_1}\, v )  . $$
The trivial fiber bundle $X \times \R^2$ admits the trivial (canonical) orientation. If the cocycle $\uA$ has periodic orbits whose product matrices have negative determinant then  $F_{\uA}$ does not preserve orientation. We are going to construct a sort of \textit{oriented double covering bundle} $\tilde X\times\R^2$ and extend $F_{\uA}$ to an orientation (and area) preserving  cocycle $F_{\tilde {\uA}}:\tilde X\times\R^2\to \tilde X\times\R^2$,
which can then can be regarded as an  $\SL_2(\R)$-valued cocycle.

We begin with the construction at the graph level.
Consider the underlying digraph $\Gscr=(\Ascr, \Ecal)$ of the fixed (mixing) sub-shift of finite type over the alphabet $\Ascr$, where $\Ecal$ is the set of admissible transitions $i\to j$ between states $i, j\in \Ascr$. Consider also a cocycle $\uA\in \SL_2'(\R)^\Ascr$.

Let $\tilde \Ascr := \Ascr \sqcup \Ascr^\ast$ (two duplicates of $\Ascr$) and  define $\tilde {\Ecal}$  in the following way.
	Each edge $(i\to j)\in \Ecal$ determines two edges of $\tilde \Ecal$ as follows:
	\begin{itemize}
	\item  $i\to j$ and $i^\ast \to j^\ast$ when $\det(A_{j})>0$;
	\item  $i\to j^\ast$ and $i^\ast \to j$ when $\det(A_{j})<0$. 	
	\end{itemize}
	Set $\tilde \Gscr:=(\tilde \Ascr, \tilde\Ecal)$, let $\tilde X$ denote the space of $\tilde \Gscr$-admissible sequences and orient 
	the trivial bundle $\tilde X \times \R^2$  in a way that
	$$ \{\omega\}\times \R^2 \, \text{ gets the } \begin{cases}
		\text{ canonical orientation } & \text{ if } \omega_0\in \Ascr  \\
		\text{ opposite orientation } & \text{ if } \omega_0\in \Ascr^\ast .
	\end{cases}  $$
	Then define $\tuA \in \SL_2'(\R)^{\tilde \Ascr}$ as follows:
	\begin{itemize}
	\item   If $\det(A_{j})>0$, both transition matrices at $i\to j$ and $i^\ast \to j^\ast$ are $\tilde A_{j} = \tilde A_{j^\ast} :=  A_{j}$ ,
	\item If  $\det(A_{j})<0$,
	both transition matrices at $i^\ast \to j$  and  $i\to j^\ast$ are $\tilde A_{j} = \tilde A_{j^\ast} :=  A_{j}$ .	  	
	\end{itemize}

  We will refer to $F_{\tuA}$ as the \textit{oriented lift} of $F_{\uA}$.
	
	\begin{remark}
	By construction the cocycle $F_{\tilde A}:\tilde X\times \R^2 \to \tilde X\times \R^2$ preserves orientation and area.
	\end{remark}

 	\begin{remark}
   Notice that as an oriented bundle $\tilde X \times \R^2$  is not a trivial bundle because it has fibers with different orientations, but  after a conjugacy that reverses orientations of fibers over $\Ascr^\ast$,  $F_{\tilde A}$ becomes an $\SL_2(\R)$-cocycle over 
   the sub-shift of finite type $\sigma:\tilde X \to \tilde X$.
 	\end{remark}

 	\begin{remark}
 	The (lifted) sub-shift  $\sigma:\tilde X \to \tilde X$ is semi-conjugated to  $\sigma:X \to X$.
	 The cocycle $F_{\tuA}$ is conjugated to  $F_A$ over the previous semi-conjugacy. In particular,
	 $F_{\uA}$ is uniformly hyperbolic if and only if	 $F_{\tuA}$ is uniformly hyperbolic. 
	\end{remark}

The graph $\tilde \Gscr$ admits an involution symmetry:

\blob\; At the level of nodes the maps $\Ascr\to \Ascr^\ast$, $i\mapsto i^\ast$ and  $\Ascr^\ast\to \Ascr$, $i^\ast\mapsto i^{\ast\ast}:=i$, give rise to an involution $\ast:\tilde \Ascr\to \tilde \Ascr$. 

\blob\; This symmetry extends to edges:
Given nodes $i, j\in \tilde \Ascr$,\; 
$(i\mapsto j)\in \tilde \Ecal$ $\Leftrightarrow$  $(i^\ast \to j^\ast)\in\tilde \Ecal$.
Hence,   defining $(i\to j)^\ast =(i^\ast\to j^\ast)$, we get an involution
$\ast:\tilde \Ecal\to\tilde \Ecal$.

\blob\; The lifted cocycle $\tuA\in \SL_2'(\R)^{\tilde \Ascr}$ preserves this symmetry in the sense that for every edge  $e=(i\to j)\in \tilde \Ecal$,
$$\tilde A(e)=\tilde A_j = A_j = \tilde A_{j^\ast} = \tilde A(e^\ast). $$

\blob\; The previous symmetry extends to an involution
$I:\tilde X\to \tilde X$, $I(\{\omega_n\}_{n\in \Z}):= \{\omega_n^\ast\}_{n\in \Z})$, that commutes with 
 $\sigma:\tilde X\to \tilde X$.

\blob\; The map $\tilde I:\tilde X\times \R^2 \to \tilde X \times \R^2 $, $ \tilde I(\omega, v):=(I(\omega), v)$,
is a  fiber bundle involution map, that commutes with $F_{\tilde A}$.

\blob\; Given a $\tilde A$-invariant multi-cone $M\subset \tilde \Ascr\times \Proj$, the set $$M^\ast:=\{(i^\ast, \hat v)\colon (i, \hat v)\in M\,\}$$
 is also an  $\tilde A$-invariant multi-cone. Notice for instance that given $(i, \hat v)\in M^\ast$ and $(i\to j)\in \tilde \Ecal$, we have
   $(i^\ast, \hat v)\in M$ and 
 $(i^\ast\to j^\ast)\in \tilde \Ecal$, which imply that
 $(j^\ast, \tilde A_{j^\ast} \hat v)\in M$. Hence 
 $$ (j, \tilde A_{j}\hat v ) = (j, \tilde A_{j^\ast} \hat v) \in M^\ast. $$

\blob\; The symmetry of $\tilde A$ implies that this cocycle admits $\tilde A$-invariant multi-cones which are invariant under the previous involution.
In fact, since the intersection of invariant multi-cones is still an invariant multi-cone, $M\cap M^\ast$ is a symmetric $\tilde A$-invariant multi-cone.

\blob\; If $\tilde M$ is a symmetric $\tilde A$-invariant multi-cone
then it determines an $A$-invariant multi-cone $M$ for $A$, defined by
$M_i:=\tilde M_i=\tilde M_{i^\ast}$ for all $i\in \Ascr$.

\begin{proof}[Proof of ABY Theorem 1 (Theorem~\ref{ABY1}  with $G=\SL_2'(\R)$)]
The non -trivial part of this theorem is the direct implication, because the converse  follows with the same argument as in~\cite{ABY10}.

The previous remarks allow us to reduce a uniformly hyperbolic cocycle $\uA\in \SL_2'(\R)^\Ascr$  to
a uniformly hyperbolic cocycle $\tuA\in \SL_2(\R)^{\tilde \Ascr}$. By Theorem~\ref{ABY1} for the case $G=\SL_2(\R)$, there exists an $\tuA$-invariant multi-cone, 
and hence also a symmetric $\tuA$-invariant multi-cone  $\tilde M$. By the last remark above, the symmetric multi-cone $\tilde M$ projects to an $\uA$-invariant multi-cone.
\end{proof}

\begin{proof}[Proof of ABY Theorem 2 (Theorem~\ref{ABY2} with $G=\SL_2'(\R)$)]
Consider the group  $G=\SL_2'(\R)$. A  key point  is that
$\uA\in \partial\, \Hscr(G)$ implies $\tuA\in \partial\, \Hscr(G)$. In fact notice that we already know that $\uA\notin \Hscr(G)$ $\Leftrightarrow$ $\tuA\notin  \Hscr(G)$, and since the construction is continuous, if $\Hscr(G)\ni \uA_n \to \uA$ then  $\Hscr(G)\ni \tuA_n \to \tuA$.

The cocycle $\tuA\in \partial \Hscr(G)$ can be regarded as 
an $\SL_2(\R)$ cocycle over the $\tilde \Ascr$-sub-shift. Thus, by Theorem~\ref{ABY2} for the group $G=\SL_2(\R)$, one the  two stated alternatives holds for $\tuA$.
Since periodic points of $\sigma:\tilde X\to \tilde X$  correspond to periodic points of $\sigma: X\to X$, and the the cocycles $F_{\tuA}$ and $F_{\uA}$ are semi-conjugated, the satisfied alternative holds for $\uA$ as well.
\end{proof}

\section{Projective uniform hyperbolicity}\label{puh}

In this Section we extend Theorem~\ref{ABY1} to $\Mat_2(\R)$-valued cocycles. 
Consider a random Markov linear cocycle $F \colon X\times \R^2\to X\times \R^2$ determined by the data $(P,q,\uA)$
where $\uA :=(A_i)_{i\in\Ascr} \in \Mat_2(\R)^\Ascr := \Mat_2(\R)^k$.

\begin{definition}
\label{def puh mat2}
We say that a linear cocycle $F$ is \textit{projectively uniformly hyperbolic} if there exists an $F$-invariant decomposition into one-dimensional subspaces  $\R^2=E_0(\omega)\oplus E_1(\omega)$ 
where  the sub-bundles $X\ni \omega \mapsto E_i(\omega)$ are continuous functions and there exists $n\in\N$ such that
$\norm{A^n\vert_{E_0}(\omega)}<\norm{A^n\vert_{E_1}(\omega)}$ for all $\omega\in X$. 
\end{definition}

Because $X$ is compact and $A\colon X\to \Mat_2(\R)$ is continuous, the last condition is equivalent to 
$$
\sup_{\omega\in X} \frac{ \norm{A^n\vert_{E_0}(\omega)} }{\norm{A^n\vert_{E_1}(\omega)}} \leq \lambda^{-1}<1 , $$  
for some $\lambda>1$.
This is further  equivalent to the existence of  $c>0$ and $\la > 1$ such that for all $\omega \in X$ and all $n \in \N$,
	$$
	\norm{A^n|_{E_1} (\omega)} \geq c\lambda^n \,  \norm{A^n|_{E_0}(\omega)}  .
	$$

Note that projective uniform hyperbolicity is also sometimes referred to as \textit{dominated splitting}.

\smallskip

Let $\Pp^1:=\Pp(\R^2)$ denote the projective line. An element of $\Pp^1$ will be denoted by $\hat v$, where $v$ is a nonzero vector (or a one dimensional subspace) in $\R^2$.  

Given an {\em invertible} matrix $A \in \Mat_2 (\R)$, its induced projection action $\hat A \colon \Pp^1 \to \Pp^1$ is given by $\hat A \hat v := \widehat{A v}$.  If $A$ has rank $1$ (that is, if it is nonzero and noninvertible), we define its projection action as the constant map $\hat A \hat v := \hat r$, where $r = \Range (A)$.   Notice that we are \textit{formally removing} the discontinuity of the projective action of a singular matrix at the kernel. 
 
Moreover, let $\conorm(A)$ denote its co-norm (i.e. smallest singular value).
If $A$ is invertible, $\conorm(A)=\norm{A^{-1}}^{-1}$,
otherwise $\conorm(A)=0$.

\begin{definition}(Bernoulli setting)
\label{multi-cone: Bernoulli}
An invariant multi-cone for $\uA=(A_i)_{i\in \mathscr{A}}$ is a nonempty set $M$ such that:
\begin{enumerate}
	\item  $M$ is an  open subset of $\Pp^1$,
	\item its closure $\bar M\neq \Pp^1$,
	\item $A_i M \Subset M$, i.e., $\overline{ A_i M } \subset M$ for every $i \in \mathcal{A}$ .	
\end{enumerate}
\end{definition}

\begin{definition}(Markov setting)
\label{multi-cone: Markov}
An invariant family of multi-cones for $\uA=(A_i)_{i\in \mathscr{A}}$ is a family $(M_i)_{i\in \mathscr{A}}$ such that for all $i,j\in\mathscr{A}$,
\begin{enumerate}
	\item  $M_i$ is an  open subset of $\Pp^1$,
	\item $\bar M_i\neq \Pp^1$,
	\item $A_i M_j \Subset M_i$, i.e., $\overline{ A_i M_j } \subset M_i$, whenever $p_{ij}>0$.	
\end{enumerate}
\end{definition}

The next theorem extends to $\Mat_2(\R)$-cocycles Yoccoz~\cite[Proposition 2]{Ycz04} and  Theorem~\ref{ABY1} (i.e.,~\cite[Theorem 2.3]{ABY10}).

\begin{theorem}
\label{thm: multi-cone charact of UH}
Given  a random linear cocycle $\uA\in \Mat_2(\R)^{\mathscr{A}}$, the following are  equivalent: 
\begin{enumerate}
	\item $\uA$ is projectively uniformly hyperbolic.
	\item There exist $c>0$ and $\lambda>1$ such that for all $n\in\N$ and $\omega\in X$, 
	$\norm{A^n(\omega)}\geq c\, \lambda^n\, \conorm(A^n(\omega))$. 
	\item $\uA$ admits an invariant multi-cone.
\end{enumerate}
\end{theorem}

\begin{proof}
$(1)\Rightarrow (2)$: Suppose that $\uA$ is projectively uniformly hyperbolic, i.e., we have dominated splitting. Then for some $c>0$ and $\lambda>1$ we have that for all $\omega\in X$ and $n\in\N$,
$$\frac{\norm{A^n(\omega)}}{\conorm(A^n(\omega))}\geq
\frac{\norm{A^n\vert_{E_1}(\omega)}}{\norm{A^n\vert_{E_0}(\omega)}} \geq c\, \lambda^n .
$$

$(2)\Rightarrow (1)$: 
 This is a straightforward adaptation of \cite[Proposition 2.1]{Viana2014} to the case where   $\uA$ takes values in $\Mat_2(\R)$. The idea of this proof is to define the Oseledets invariant directions $E_0(\omega)$ and $E_1(\omega)$ as uniform limits of continuous functions, by exploiting the contracting behavior of the cocycle action on fibers due to the cocycle's hyperbolicity.
 This implies that the Oseledets splitting is continuous.
 In our setting, when singular matrices appear,  their actions contract the whole projective  space to points.
 This helps the convergence and poses no problem regarding the continuity of the approximations,  which are defined as singular directions of the the matrices $A^n(\omega)$ and
 $A^n(\sigma^{-n}\omega)$.

\smallskip

For the remaining implications we  need a
\textit{desingularization construction} that associates a family of invertible cocycles
$\uA^\ast(\mu) \in \SL_2'(\R)$ to every singular cocycle $\uA\in \Mat_2(\R)^k$.
For each $A_i \in \Mat_2(\R)$ with $i \in \mathcal{A}$,   consider its singular value decomposition: $A_i = R_i \Sigma_i R^\ast_i$. If $A_i$ is not invertible, consider a small perturbation of $\Sigma_i$ that transforms its zero singular value into a small constant $\mu^{-2}$. In other words, if $A_i= R_i \begin{bmatrix}
	\norm{A_i} & 0 \\
	0 & 0
\end{bmatrix}R_i^\ast$, let 
$$\tilde{A}_{\mu, i}:= R_i \begin{bmatrix}
	\norm{A_i} & 0 \\
	0 & \mu^{-2}
\end{bmatrix}R_i^\ast = \frac{\norm{A_i}}{\mu} R_i \begin{bmatrix}
	\mu & 0 \\
	0 & \mu^{-1}
\end{bmatrix}R_i^\ast \, .$$
If $A_i$ is invertible put $\tilde A_{\mu, i} := A_i$.  Then set 
$A_{\mu,i}^\ast:= \frac{1}{\sqrt{\det (\tilde{A}{\mu, i})}}\, \tilde A_{\mu,i} $. 
 The cocycle $\uA_{\mu}^\ast =  (A_{\mu , i}^\ast)_{i \in \mathcal{A}} \in \SL_2'(\R)^k$   has the same projective action as $(\tilde A_{\mu , i} )_{i \in \mathcal{A}}$, which for large $\mu$ approximates that of $\uA$.

\medskip

\begin{lemma}
\label{desingularization lemma}
Let $\uA\in \Mat_2(\R)^k$.
\begin{enumerate}
	\item If $(M_i)_{i\in \Ascr}$ is an invariant multi-cone of $\uA^\ast_\mu$,  for some $\mu>0$, then it is also an invariant multi-cone for $\uA$.
	\item If $(M_i)_{i\in \Ascr}$ is an invariant multi-cone of $\uA$ then it is also an invariant multi-cone of $\uA^\ast_\mu$ for all sufficiently large $\mu$.
\end{enumerate}
\end{lemma}

\begin{proof} For   simplicity we only address the Bernoulli case.
	
(1)\; Since $M$ is an invariant  multi-cone for   $\uA_\mu^\ast$,   $\tilde A_{\mu ,i} M \Subset M$ for every $i\in \Ascr$. Moreover, the pair of matrices $A_i$ and $\tilde A_{\mu ,i}$ share the same singular directions, but their contraction strengths are different (infinite in the case of non-invertible matrices vs. finite for invertible ones). Thus since the contraction is stronger for non-invertible matrices, we conclude that $A_i M \subseteq A_{\mu ,i} M  \Subset M$ for every $i\in \Ascr$.

(2)\; This holds since multi-cones are stable under perturbations and $\uA=\lim_{\mu\to \infty} \underline{\tilde A}_\mu$, while  $\uA^\ast_\mu$ and $\underline{\tilde A}_\mu$ share their invariant multi-cones, because they induce the same action on $\Proj$.
\end{proof}

We return to the proof of the theorem.

$(1)\Rightarrow (3)$: If $\uA$ is projectively uniformly hyperbolic then the approximating cocycles $\underline{\tilde A}_\mu$ and $\uA_\mu^\ast$ are also projectively uniformly hyperbolic for large $\mu$. By  Theorem~\ref{ABY1} 
the cocycle $\uA_\mu^\ast$ 
admits an invariant multi-cone $M$. Therefore,  by (1) of Lemma~\ref{desingularization lemma}, $M$ is also an invariant multi-cone for $\uA$.

$(3)\Rightarrow (2)$: Suppose $\uA$ admits an invariant multi-cone $M=(M_i)_{i \in \mathcal{A}}$. By (2) of Lemma~\ref{desingularization lemma}, for some large $\mu$,
 $M$ is also an invariant multi-cone for  $\uA_\mu^\ast$.
Therefore, by Theorem~\ref{ABY1}
 the cocycle $\uA_\mu^\ast$ is uniformly hyperbolic and by \cite[Proposition 2.1]{ABY10} there exists $c>0$ and $\lambda>1$ such that
$\norm{ (A^\ast_\mu)^n(\omega) }\geq c\, \lambda^n$ for all $\omega \in X$, which in turn implies that 
$$ \frac{ \norm{A^n(\omega)} }{\conorm (A^n(\omega))}\geq \frac{ \norm{\tilde A_\mu^n(\omega)} }{\conorm (\tilde A_\mu^n(\omega))} = \norm{ (A^\ast_\mu)^n(\omega) }^2 \geq c\, \lambda^{2n} \quad \forall\; \omega \in X . $$

Note that either $A^n(\omega)=A_{\omega_{n}}\, \cdots \, A_{\omega_1}$  does not include singular matrices so
$A^n(\omega)= \tilde A_\mu^n(\omega)$, or else it does
and then $\conorm(A^n(\omega))=0$, which implies that the left-hand side of the above inequality is $\infty$.
\end{proof}

The rest of this Section is devoted to prove another characterization of PUH. (Theorem~\ref{thm2: branch charact of UH})
that will play a fundamental role in proving the dichotomic behavior of the Lyapunov exponent in the next section.

\begin{definition}
We say that $\uA\in \Mat_2(\R)^{\mathscr{A}}$ has rank $1$ if 
$$\lim_{n\to \infty} \rank(A^n(\omega))=1\quad   \text{ for } \; \mu \text{-almost all }\, \omega\in X. $$ 
\end{definition}

We do not consider cocycles of rank $0$ because in this case the first Lyapunov exponent is equal to $-\infty$. 
Cocycles of rank $2$, that is, in $\GL_2(\R)^{\Ascr}$, are also not considered here because,  as mentioned in the introduction, they have already been extensively studied. Thus, from now on we only consider cocycles $\uA$ of rank $1$.

\begin{remark} It is easy to verify that a cocycle $\uA\in \Mat_2(\R)^{\mathscr{A}}$ has rank $1$ if and only if
\begin{enumerate}
	\item $\rank(A_i)=1$ for some $i\in \mathscr{A}$ and
	\item $\uA$ has no null word, i.e. $A^n(\omega):= A_{\omega_{n}}\, \cdots \, A_{\omega_2}\, A_{\omega_1}\neq 0$ for every $\omega\in X$.
\end{enumerate}	
\end{remark}

We split the alphabet $\Ascr$ into two parts: 

\quad $\Ainv:=\{i\in \Ascr\colon \det A_i\neq 0\}$ and
$\Asing:=\{i\in \Ascr\colon \rank A_i = 1\}$.

By a slight abuse of notation, in the  Markov setting we also write
$\Amap_i \colon \Ascr\times\Proj \to \Ascr\times\Proj$ to denote,  for  each $i\in \Ascr$, the non-invertible map
$\Amap_i(j,\hat v):=  (i,\hat A_i\, \hat v)$. 

\begin{remark}
\label{Markov multi-cone remark}
In the Markov case the family $(M_i)_{i\in\Ascr}$  of multi-cones determines the set $M:=\cup_{i \in \Ascr}\{i\}\times M_i$ which has the  properties:
\begin{enumerate}
	\item  $M$ is an  open subset of $\Ascr\times \Pp^1$,
	\item the closure of each fiber of $M$ is a proper
	subset of $\{i\}\times \Proj$,
	\item $\Amap_i M \Subset M$  for every $i \in \mathcal{A}$ ,	
\end{enumerate}
and we refer to $M$ as an invariant multi-cone of $\uA$.
\end{remark}

For $i\in\Asing$, write $r_i:=\Range(A_i)$ and $k_i:=\Ker(A_i)$ and set
\begin{align*}
	\Kscr(\uA)&:=\{(j,k_i)\colon   i\in\Asing, j \in \Ascr \} \ ,\\
	\Rscr(\uA)&:=\{ (i,r_i) \colon i\in \Asing\} .
\end{align*}
The complement $\Ascr\times\Proj\setminus \Kscr(\uA)$  is the common domain of all maps $\Amap_i$ with $i\in \Asing$, while $\Rscr(\uA)$  is the union of the ranges of these same maps.
In the Bernoulli case we simply define
\begin{align*}
\Kscr(\uA)&:=\{ k_i\colon i\in \Asing\} ,\\
\Rscr(\uA)&:=\{ r_i \colon i\in \Asing\} .
\end{align*}

Given $i\in \Asing$, a \textit{branch departing} from $i$  
is any $P$-admissible word $\underline \omega=(\omega_0, \omega_1, \cdots ,\omega_n)\in \Ascr^{n+1}$ with $n\geq 0$ such that  $\omega_0=i$ and $\omega_l\in \Ainv$ for all $l=1,\ldots, n$.
Similarly, a \textit{branch arriving} at $i$  
is a $P$-admissible word $\underline \omega=(\omega_0, \omega_1, \cdots ,\omega_n)\in \Ascr^{n+1}$ with $n\geq 0$ such that  $\omega_n=i$ and $\omega_l\in \Ainv$ for all $l=0,\ldots, n-1$.
Denote by $\Bscr_n^+(i)$ the set of all branches $\underline \omega=(\omega_0, \omega_1, \cdots ,\omega_n)\in \Ascr^{n+1}$  departing from $i$ and  by $\Bscr_n^-(i)$ the set of all branches $\underline \omega=(\omega_0, \omega_1, \cdots ,\omega_n)\in \Ascr^{n+1}$   arriving at $i$.
For $\underline\omega\in \Bscr_n^+(i)$ we write
$A^{n-1}(\underline\omega):=A_{\omega_{n-1}}\, \ldots\, A_{\omega_2}\, A_{\omega_1}$ while for $\underline\omega\in \Bscr_n^-(i)$ we write
$A^{-(n-1)}(\underline\omega):=A_{\omega_{1}}^{-1}\, A_{\omega_2}^{-1} \, \ldots\, A_{\omega_{n-1}}^{-1}=(A_{\omega_{n-1}}\, \ldots\, A_{\omega_2}\, A_{\omega_1})^{-1}$.
These matrices are in\-ver\-tible by definition of a branch. 

In the Bernoulli case all sequences $\underline\omega$ are admissible and we set  
\begin{align*}
\Wscr^+&:=\overline{\cup_{i\in \Asing} \cup_{n\geq 0} \left\{ A^{n-1}(\underline \omega)\, r_i,\; \underline \omega\in \Bscr_n^+(i) \right\}} ,\\
\Wscr^-&:=\overline{\cup_{i\in\Asing}\cup_{n\geq 0} \left\{ A^{-(n-1)}(\underline \omega)\, \kappa_i,\; \underline \omega\in \Bscr_n^-(i) \right\}} .
\end{align*} 
In the Markov case the definition of these sets is analogous.
The subset $\Wscr^+\subset \Ascr\times \Proj$ is defined to be the closure of the union of the ranges of compositions of the partial maps $\Amap_{\omega_l}$ along all branches departing from $i$ while the subset $\Wscr^-\subset \Ascr\times \Proj$ is the closure of the union of the pre-images of $0$ under compositions of the partial maps $\Amap_{\omega_l}$ along all branches arriving at $i$.

\begin{definition}
In the Bernoulli case, if $M$ is an invariant multi-cone for the invertible cocycle $(A_i)_{i\in \Ainv}$,   we define the sets $\Kinv^u$ and $\Kinv^s$ as in ~\cite[Subsection 2.3]{ABY10}, i.e., 
$$\Kinv^u=\bigcap_{n=0}^{\infty} \, \, \bigcup_{i_1,\cdots,i_n\in\Ainv}A_{i_n}\cdots A_{i_1}(M)$$ and 
$$\Kinv^s=\bigcap_{n=0}^{\infty} \, \,  \bigcup_{i_1,\cdots,i_n\in\Ainv}(A_{i_n}\cdots A_{i_1})^{-1}(\Pp^1\setminus\overline{M}).$$
In the Markov case $\Kinv^u$ and $\Kinv^s$ are subsets of $\Ascr\times \Proj$. Firstly, the multi-cone is to be interpreted according to Remark~\ref{Markov multi-cone remark}  and $\Pp^1\setminus\overline{M}$ replaced by
$\Ascr\times \Proj\setminus \overline{M}$. Secondly, the matrix products $A_{i_n}\cdots A_{i_1}$ are to be substituted by  composition of the maps
$\Amap_{i_n}\circ \cdots \circ \Amap_{i_1}$
along admissible invertible words.
\end{definition}

\begin{proposition}
\label{Ku=Eu, Ks=Es}
If $M$ is an invariant multi-cone for the invertible cocycle $(A_i)_{i\in \Ainv}$ then:
\begin{enumerate}
	\item $\Kinv^u$ is  the  set of  unstable Oseledets directions $E^u(\omega)$ of the cocycle
	$(A_i)_{i\in \Ainv}$ over the set of points $\omega\in \Ainv^\Z$,
	\item $\Kinv^s$ is the set of  stable Oseledets directions $E^s(\omega)$ of the cocycle
$(A_i)_{i\in \Ainv}$ over the set of  points $\omega\in\Ainv^\Z$.	
\end{enumerate}
\end{proposition}

\begin{proof}
Fix $\hat w\in M$. Then  by dominated splitting,  $\hat v \in \Kinv^u$ is the limit
$$ \hat v  = \lim_{n\to \infty} A_{\omega_{-1}}\, \cdots\, A_{\omega_{-n}}\, \hat w  = \lim_{n\to \infty} A^n(\sigma^{-n}\omega)\, \hat w = E^u(\omega) , $$
for some sequence $\omega\in\Ainv^\Z$. Similarly,  $\hat v \in \Kinv^s$ is the limit
$$ \hat v  = \lim_{n\to \infty} ( A_{\omega_{n-1}}\, \cdots\, A_{\omega_{0}} )^{-1} \, \hat w  = \lim_{n\to \infty} A^{-n}(\sigma^{n}\omega)\, \hat w = E^s(\omega) , $$
for some sequence $\omega\in\Ainv^\Z$.
\end{proof}

For the rest of this section all stated results hold in the Markov case but for simplicity the proofs
  will be done in the Bernoulli  context.
 The extension of these proofs to the Markov case
 are simple adaptations using the notation introduced above.


\begin{proposition}\label{reducing the multi-cone}
Let $M$ be an invariant multi-cone of the invertible cocycle
$(A_i)_{i\in \Ainv}$  and take $v\in M\setminus \Kinv^u$. Then there exists an invariant multi-cone $\tilde{M}$ of $(A_i)_{i\in \Ainv}$ such that $ \tilde M \Subset M$ and  $v\notin\tilde{M}$. 
\end{proposition}

\begin{proof}
 Define $$M_n:=\bigcup_{|\underline{\omega}|=n} A^n(\underline{\omega})M $$
 where the union is taken over all admissible  invertible  words of length $n$. We claim that $M_n$ is an invariant multi-cone of $(A_i)_{i\in \Ainv}$, for every $n\in\mathbb{N}$ and that there exists a sufficiently large $N\in\mathbb{N}$ such that $v\notin M_N$. Let us prove, by induction, that $M_{n+1}\Subset M_n\Subset\cdots \Subset M$. 
Since $M$ is an invariant multi-cone associated to $(A_i)_{i\in \Ainv}$, then $A(\underline{\omega})M\Subset M$ for every invertible word $\underline{\omega}$ such that $|\underline{\omega}|=1$. Thus $\bigcup_{|\underline{\omega}|=1}A(\underline{\omega})M\Subset M$ as the union is finite. Then
\begin{align*}
M_{n+1}&= \bigcup_{|\underline{\omega}|=n+1}A^{n+1}(\underline{\omega})M = \bigcup_{|\underline{z}|=1}A(\underline{z})\bigcup_{|\underline{\omega}|=n}A^{n}(\underline{\omega})M\\
&= \bigcup_{|\underline{z}|=1}A(\underline{z})M_n \Subset M_n.
\end{align*}
In particular, as $M\neq \Pp^1$ we have that $M_n\neq \Pp^1 \ \forall n\in\mathbb{N}$. Notice that as $M$ is open and the cocycle $(A_i)_{i\in \Ainv}$ is invertible, then $M_n$ is open for every $n\in\mathbb{N}$. Therefore $M_n$ is an invariant multi-cone $\forall n\in\mathbb{N}.$ To prove that
for any $v\in M\setminus \Kinv^u $ there exists $n\in\mathbb{N}$ such that $v\notin M_N$ we simply notice that $\Kinv^u$ is closed and that $$\lim_n M_n := \lim_n \bigcup_{|\underline{\omega}|=n}A^{n}(\underline{\omega})M = \bigcap_{n=0}^{\infty} \bigcup_{|\underline{\omega}|=n}A^{n}(\underline{\omega})M =: \Kinv^u ,$$ 
because $M_n$ is a monotonous sequence.
\end{proof}

\begin{proposition} \label{PropEmptyInt}
Consider  a cocycle $\uA\in \Mat_2(\R)^{\mathscr{A}}$ of rank $1$ such that $\uA_{\rm inv}:=(A_i)_{i\in \Ainv}$ is projectively uniformly hyperbolic, $\underline{A}$ is not diagonalizable and $ \Wscr^+\cap \Wscr^-=\emptyset $. Then
$$
\Kscr(\uA) \cap \Kinv^u = \emptyset \quad \text{and} \quad  \Rscr(\uA) \cap \Kinv^s = \emptyset.
$$
\end{proposition}

\begin{proof}
We are only going to prove that  $\Kscr(\uA)\cap \Kinv^u = \emptyset$. The other proof is analogous. Suppose by contradiction that there exists $\hat k \in \Kscr(\uA)\cap \Kinv^u$. We split the proof into two cases:\\
\blob\; $\uA_{\rm inv}$ is not diagonalizable, and \\
\blob\;  $\uA_{\rm inv}$ is diagonalizable but $\uA$ is not.

Let us start with the assumption that $\uA_{\rm inv}$ is not diagonalizable.
We will say that $\hat v\in\Proj$ is $\uA_{{\rm inv}}$-invariant 
if $ A_i\, \hat v=\hat v$ for all $i\in \Ainv$.

\begin{lemma} \label{InvElemKu}
There exists an $\uA_{{\rm inv}}$-invariant element in $\Kinv^{u}$ if and only if $\# 	\Kinv^{u} = 1$. Analogously, there exists an $\uA_{{\rm inv}}$-invariant element in $\Kinv^{s}$ if and only if $\# \Kinv^{s} = 1$.
\end{lemma}

\begin{proof}
Suppose $\#\Kinv^u = 1$. By Proposition~\ref{Ku=Eu, Ks=Es}, there exists $\hat v \in \Kinv^u$ such that for every $\omega \in X$,  
$E^u(\omega) = \hat v$. Moreover, by the invariance of the Oseledets subspaces, for almost every $\omega \in X$, $E^u(f(\omega)) = A(\omega)E^u(\omega)$. Hence, $\hat v = \hat A(\omega)\hat v$ for almost every $\omega \in X$.

Conversely, if $A_i\, \hat v=\hat v\in \Kinv^u$ 
for all $i\in \Ainv$, since $\hat v\in \Kinv^u\subset M$, 
$$ \hat v  = \lim_{n\to \infty} A_{\omega_{-1}}\, \cdots\, A_{\omega_{-n}}\, \hat v  = \lim_{n\to \infty} A^n(\sigma^{-n}\omega)\, \hat v = E^u(\omega) , $$
and $E^u(\omega)=\hat v$ for all $\omega\in \Ainv^\Z$.
Thus by Proposition~\ref{Ku=Eu, Ks=Es} $\Kinv^u =\{\hat v\}$.

The conclusion for $\Kinv^s$ follows  under a similar argument.
\end{proof}

By Lemma~ \ref{InvElemKu},   $\#\Kinv^u >1$ or $\#\Kinv^s>1$,
for otherwise $\uA_{\rm inv}$ would be diagonalizable. We treat each of these cases separately. 

$\blob$ First   assume that $\#\Kinv^u>1$.

If there exists $\hat r \in \Rscr(\uA)$ such that $\hat r \notin \Kinv^s$, then for every $\varepsilon > 0$, we can choose $\omega \in \Ainv^\Z$  and $n\in\N$ such that $d( A^n(\omega)\, \hat r, \hat k) < \varepsilon$. This contradicts  the fact that $ \Wscr^+\cap \Wscr^-=\emptyset $.
		
 On the other hand, if $\hat r \in \Rscr(\uA) \cap \Kinv^s$, since $\Kinv^u$ has at least two elements one of which is $\hat k\in \Kinv^u$, by Lemma~\ref{InvElemKu} there exists $i\in\Ainv$ such that
$A_i^{-1}\hat k\neq \hat k$. Hence there exists a word    $\omega \in \Ainv^\Z$ and $n_0\in\N$ such that $A^{-n_0}(\omega)\, \hat k \notin \Kinv^u$. Choosing the coordinates of $\omega$ appropriately we can force the convergence of  $A^{-n}(\omega)\, \hat k$    to $\hat r \in \Kinv^s$, which also contradicts the fact that $ \Wscr^+\cap \Wscr^-=\emptyset $. 

\smallskip
	
$\blob$ Now assume $\#\Kinv^s > 1$. 

Suppose that there exists  $\hat r \in\Rscr(\uA)$ such that $\hat r \notin \Kinv^s$. Then it is possible to iterate $\hat r$ forward by a suitable invertible word in a way that it converges  to $\hat k \in \Kinv^u$. This contradicts $ \Wscr^+\cap \Wscr^-=\emptyset $. 

 If, on the other hand, there exists $\hat  r \in \Rscr(\uA)\cap \Kinv^s$, then by Lemma~\ref{InvElemKu} there is $i\in\Ainv$ such that
		$A_i\hat r\neq \hat r$. Hence there exists a word    $\omega \in \Ainv^\Z$ and $n_0\in\N$ such that
		$A^{n_0}(\omega)\, \hat r \notin \Kinv^s$. 
		Finally, choosing the coordinates of $\omega$ appropriately we can force the convergence of  $A^{n}(\omega)\, \hat r$    to $\hat k \in \Kinv^u$, which also contradicts the fact that $ \Wscr^+\cap \Wscr^-=\emptyset $.

\medskip

We now proceed to the case that $\uA_{\rm inv}$ is diagonalizable but $\uA$ is not. Since $\uA_{\rm inv}$ is projectively uniformly hyperbolic and $\uA_{\rm inv}$ is diagonalizable, there are exactly two invariant directions,   $\Kinv^s$ and $\Kinv^u$.

\begin{lemma} \label{Ks=1Ku=1}
Suppose $\uA_{\rm inv}$  is diagonalizable but $\uA$ is not. 
Then either
\begin{itemize}
\item[(i)] there exists $ \hat r \in \Rscr(\uA) \; \text{such that }\; \hat r \notin \Kinv^s \cup \Kinv^u $, or else
\item[(ii)] there exists $ \hat  k \in \Kscr(\uA) \; \text{such that }\; \hat k \notin \Kinv^s \cup \Kinv^u$.
\end{itemize}

\end{lemma}

\begin{proof}
Since  $\uA_{\rm inv}$  is diagonalizable, $\Kinv^u=\{\hat e^u\}$ and $\Kinv^s=\{\hat e^s\}$ are singletons  where 
$\hat e^u$ and $\hat e^s$ are respectively the unstable and stable directions of Oseledets. 
If (i) and (ii) were both false, then for every $i\in \Asing$,
the matrix $A_i$ would preserve both directions 
$\hat e^u$ and $\hat e^s$, which would imply that  $\uA$ is diagonalizable.
\end{proof}

Next we analyze the two cases given by Lemma~\ref{Ks=1Ku=1}.

\medskip

If  there exists $\hat r \in \Rscr(\uA)$ such that $\hat r \notin \Kinv^s$ then iterating $\hat r$ by any invertible word, it converges to $\Kinv^u =\{\hat k\}$. This contradicts $ \Wscr^+\cap \Wscr^-=\emptyset $.
Otherwise, every $\hat r\in\Rscr(\uA)$ satisfies $\Kinv^s=\{\hat r\}$ and  there exists $ \hat k' \in \Kscr(\uA)$   such that $\hat k' \notin \Kinv^s \cup \Kinv^u$. 
 Hence iterating $\hat k'$ backwards by any invertible word, it converges to $\hat r = \Kinv^s$, which contradicts $ \Wscr^+\cap \Wscr^-=\emptyset $.
 This concludes the proof.
\end{proof}

\begin{remark} \label{rmk prop puh}
Note that in   Proposition~\ref{PropEmptyInt},   assumption $\Wscr^+\cap \Wscr^-=\emptyset$ can be replaced by $\Wscr^+\cap \Kscr(\uA)=\emptyset$ and $\Rscr(\uA)\cap \Wscr^-=\emptyset$.
\end{remark}

\medskip

The following characterization of PUH provides a technical step in the proof of Theorem~\ref{thm: BM type dichotomy}.
A result in the same direction is proved in the  preprint~\cite{CSZZ-25}.

\begin{theorem} 
	\label{thm2: branch charact of UH}
	Given  a random cocycle $\uA\in \Mat_2(\R)^{\mathscr{A}}$ of rank $1$ such that $\uA_{\rm inv}:=(A_i)_{i\in \Ainv}$ is projectively uniformly hyperbolic and $\uA$ is not diagonalizable, the following are  equivalent: 
	\begin{enumerate}
		\item $\uA$ is projectively uniformly hyperbolic,
		\item $\Wscr^+\cap \Wscr^-=\emptyset$,
		\item $\Wscr^+\cap \Kscr(\uA)=\emptyset$ and $\Rscr(\uA)\cap \Wscr^-=\emptyset$.
	\end{enumerate}
\end{theorem}

\begin{proof}[Proof of Theorem~\ref{thm2: branch charact of UH}]
(1) $\Rightarrow$ (2): If $\uA$ is projectively uniformly hyperbolic, then all the statements in Theorem \ref{thm: multi-cone charact of UH} hold. Since these are open properties, meaning that they hold under arbitrarily small perturbations of the cocycle,  $\Wscr^+\cap \Wscr^-=\emptyset$, for otherwise we could produce null words under arbitrarily small perturbations of the cocycle's matrices. This would imply the loss of the projective uniform hyperbolicity, leading to a contradiction, as the projective uniform hyperbolicity is an open property.

\medskip

Note that (2) $\Rightarrow$ (3) follows from Proposition~\ref{PropEmptyInt}. 

\bigskip

(3) $\Rightarrow$ (1): 
We know that $\uA_{\rm inv}$ is projectively uniformly hyperbolic so by Theorem \ref{thm: multi-cone charact of UH}, there is an invariant multi-cone $M$ associated to $\uA_{\rm inv}$. Note that by Proposition \ref{PropEmptyInt} and Remark \ref{rmk prop puh} there are no ranges of singular matrices in $\Kinv^s$ and there are no kernels of singular matrices in $\Kinv^u$. Also, by Proposition \ref{reducing the multi-cone}, we can shrink the multi-cone $M$ so that it does not contain kernels of  singular matrices $A_i$, $i\in \Asing$. Now, because $\uA_{\rm inv}$ is projectively uniformly hyperbolic, there exists $N\in\N$ such that for every $\omega\in\Ainv^\Z$ and every range $\hat r\in\Rscr(\uA)$, which as we have previously seen is not in $\Kinv^s$, we have $A^N(\omega)\, \hat r\in M$. Thus, because there are only finitely many ranges in $\Rscr(\uA)$ and finitely many words of length $N$, using (3), we can find sufficiently small numbers $$0<\epsilon_1 < \cdots < \epsilon_{N-1}$$ independent of the words $\omega$ and of the ranges $\hat r$ such that the following inclusions of balls in the projective space hold:

$$  
A_{\omega_N} B_{\epsilon_{N-1}}(A^{N-1}(\omega)\, \hat r) \Subset M 
$$
and
$$A_{\omega_j} B_{\epsilon_j}(A^j(\omega)\, \hat r)\Subset B_{\epsilon_{j+1}}(A^{j+1}(\omega)\, \hat r)) ,$$ for every $1\leq j\leq N-1$.
The union of $M$ with all these balls, for every word $\omega\in \Ainv^N$ and every range $\hat r\in \Rscr(\uA)$  is an invariant multi-cone associated to $\uA$.
\end{proof}	
	
\begin{remark} \label{RmkNonDiag}
If $\uA_{\rm inv}$ is projectively uniformly hyperbolic and the cocycle $\uA$ is diagonalizable with 
$$\Kinv^u=\Rscr(\uA)=\{\hat r\}\; \text{ and }\; \Kinv^s=\Kscr(\uA)=\{\hat k\},$$
 then the equivalences in  Theorem \ref{thm2: branch charact of UH} hold, i.e.,   (1), (2) and (3) are all true.
  The non diagonalizable hypothesis aims to exclude the case when \begin{equation}
  	\label{diag relations}
  	\Kinv^u=\Kscr(\uA)=\{\hat k\}\, \text{ and }\;  \Kinv^s=\Rscr(\uA)=\{\hat r\} .
  \end{equation}
  In Example~\ref{explo 1} the assumptions of this theorem hold, except for the fact that $\uA$ is diagonalizable and satisfies~\eqref{diag relations}. In this example conditions (1) and (3) of Theorem \ref{thm2: branch charact of UH} fail while condition (2) holds true.
\end{remark}	

\medskip

\section{Continuity dichotomy}\label{dichotomy}
In this section we prove a Ma\~{n}\'e-Bochi theorem for non-invertible $\Matdm$-valued cocycles. As a corollary, we obtain a dichotomy between analyticity and discontinuity for the Lyapunov exponents.

Let $\uA:\Lambda \to \Matdm^k$ be a smooth family on some open set $\Lambda\subset \R^n$.

\begin{definition}
\label{def pos winding}
When $\Lambda\subset \R$ is an open interval we say that  
$\uA(t)$ is \textit{positively winding} if there exists $c>0$ such that for all $j\in \Ainv$,
\begin{itemize}
	\item[$\blob$] all $t\in \Lambda$ and  $v\in \Proj$,
\begin{equation}
	\label{winding speed}
 \frac{(A_j(t)\, v)\wedge (\frac{d}{dt} A_j(t)\, v)}{\norm{A_j(t) \, v}^2} = \frac{A_j(t)\, v}{\norm{A_j(t) \, v}} \wedge \frac{d}{dt} \frac{A_j(t)\, v}{\norm{A_j(t) \, v}} \geq c ,
\end{equation}
\end{itemize}	
and moreover, for all $i\in \Asing$ and all $t\in \Lambda$, 
\begin{itemize}
	\item[$\blob$] $\rank(A_i(t))=1$,  
	
	\item[$\blob$] $r_i(t)\wedge \frac{d}{dt} r_i(t) \geq 0$, with $r_i:=\mathrm{range}(A_i)$,
	
	\item[$\blob$]  $k_i(t)\wedge \frac{d}{dt} k_i(t) \leq 0$, with $k_i:=\ker(A_i)$.
\end{itemize}
Analogously, we say that $\uA(t)$ is \textit{negatively winding} if it satisfies similar conditions  with all inequalities reversed, including $c<0$.
\end{definition}

\begin{remark}
\label{rmk pos winding examples}
Let
$R_t:=\begin{bmatrix}
	\cos t & -\sin t \\ \sin t & \cos t 
\end{bmatrix}$. Given $\uA=(A_1, \ldots, A_k) \in \Matdm^k$ the families $\uA:\R\to \Matdm^k$ below are positively winding:
\begin{itemize}
	\item[$\blob$] $\uA(t)$ where $A_i(t):=R_t\, A_i$ for all $i\in \Ascr$;
	\item[$\blob$] $\uA(t)$ where $A_i(t):= A_i\, R_t$ for all $i\in \Ascr$;
	\item[$\blob$] $\uA(t)$ where $A_i(t):=\begin{cases}
		R_t\, A_i & \text{ if } i\in \Ainv\\
		A_i & \text{ if } i\in \Asing
	\end{cases}$.
\end{itemize}
\end{remark}

\begin{remark}
	The quantity~\eqref{winding speed} above measures the speed of the projective curve $\Lambda\ni t\mapsto A_j(t)\, v\in\Proj$.
	Moreover, if $\uA(t)$ is positively winding and invertible then 
	the inverse cocycle $\uA(t)^{-1}$ is negatively winding.
\end{remark}

\begin{remark}
By the chain rule (see also~\cite[Definition 1.1, Proposition 2.4]{BCDFK}), if    $\uA(t)$ is positively winding then
for all $\omega\in \Ainv^n$, $n\in\N$  and  $i\in \Asing$,
$$ \frac{(A^n_t(\omega)\, r_i(t) )\wedge (\frac{d}{dt} A^n_t(\omega)\, r_i(t))}{\norm{A^n_t(\omega) \, r_i(t)}^2} \geq c ,$$
and
$$	\frac{(A^n_t(\omega)^{-1}\, k_i(t) )\wedge (\frac{d}{dt} A^n_t(\omega)^{-1}\, k_i(t))}{\norm{A^n_t(\omega)^{-1} \, k_i(t)}^2} \leq -c .$$
\end{remark}

\begin{remark}
Let	 $A:\R\to \SL_2'(\R)$ be any smooth curve and $v\in \R^2$  a  non zero vector
such that $\hat A(t_0)\, \hat v=\hat v$ and  $A(t_0)$ is an involution (i.e., it has eigenvalues $\pm 1$). An easy calculation shows that
$$\frac{d}{dt}\left[  \frac{A(t)^2\, v}{\norm{A(t)^2 \, v}} \right]_{t=t_0}=0 .$$
Hence no $\SL_2'(\R)$-valued cocycle with values of negative determinant can be positively winding in any reasonable sense.	
This explains why the next result does not hold for general $\Mat_2(\R)$-valued cocycles.
\end{remark}

\medskip

\begin{definition}
We say that $\uA:\Lambda\to \Matdm^k$ is a \textit{rich family} if for every  $p_0\in \Lambda$ there exists a smooth curve $I\ni t\mapsto \alpha(t)\in\Lambda$  passing though $p_0$ such that the $1$-parameter family $\uA(\alpha(t))$ is positively winding.
\end{definition}

\bigskip

\begin{theorem} \label{ManeBochiMat2}
Let  $\uA \colon \Lambda \to \Matdm^k$ be a  rich smooth family of cocycles on some open set $\Lambda \subseteq \R^m$ and assume that there exists  $i \in \{ 1, \dots ,k \}$ such that $\rank(A_i(t))=1$  for every $t \in \Lambda$. Then given $t \in \Lambda$, either $\uA(t)$ is projectively uniformly hyperbolic or else there exists a sequence $t_n \to t$ in $\Lambda$ such that
$\uA(t_n)$ has a periodic null word, and in particular  $L_1(\uA(t_n)) = -\infty$, for every $n \in \N$.
\end{theorem}

\begin{proof}
Suppose that $\uA(t)$ is not projectively uniformly hyperbolic. We are going to show that there exists a sequence $t_n \to t$ in $\Lambda$ such that has a periodic null word for every $n$.  Because $\uA(t)$ is rich we may immediately assume that $\dim\Lambda=1$ and the family is positively winding.  We divide the proof in three cases, according to  Theorem~\ref{thm2: branch charact of UH}.

\textbf{Case 1}: Suppose that $\uA_{\rm inv}(t)$ is projectively uniformly hyperbolic and $\uA(t)$ is not diagonalizable. Then, by Theorem~\ref{thm2: branch charact of UH} \, $\Wscr^+\cap \Wscr^- \neq \emptyset$.
 I.e., there exist  $i,j\in \Asing$ and sequences of finite words $\omega_k$ and $\omega_k'$ of sizes $n_k$ and $n_k'$, respectively, with $\lim_{k\to \infty} d(A^{-n_k'}_t(\omega_k')\, k_i,\, A^{n_k}_t(\omega_k)\, r_j)=0.$ Hence, since the projective curves
 $t\mapsto A^{-n_k'}_t(\omega_k')\, k_i$ and
 $t\mapsto A^{n_k}_t(\omega_k)\, r_j$ wind in opposite directions, with a speed uniformly bounded from below, 
 there is a sequence of parameters $t_k\to t$ in $\Lambda$ where
 $A^{-n_k'}_{t_k}(\omega_k')\, k_i=A^{n_k}_{t_k}(\omega_k)\, r_j$.
 This implies    $A_i\, A^{n_k'+n_k}_{t_k}(\omega_k'\,\omega_k)\, A_j=0$  and  hence $L_1(\uA(t_k)) = -\infty$.

\textbf{Case 2}: Suppose that $\uA_{\rm inv}(t)$ is projectively uniformly hyperbolic and $\uA(t)$ is diagonalizable. If $\Kinv^u=\Rscr(\uA)$ and $\Kinv^s=\Kscr(\uA)$, then, by Remark~\ref{RmkNonDiag}, Theorem~ \ref{thm2: branch charact of UH} still holds. Therefore $\Kscr(\uA)\cap \Wscr^+ = \emptyset$ and $\Rscr(\uA)\cap \Wscr^- = \emptyset$, which by this theorem implies that $\uA(t)$ is projectively uniformly hyperbolic. 

Hence it suffices to consider the case in which $\Kinv^u= \Kscr(\uA)$ and $\Kinv^s=\Rscr(\uA)$. 
Since these sets are singletons,  increasing $t$ we move the ranges of singular matrices  out of $\Kinv^s$ and at the same time the kernels of these matrices  move away from $\Kinv^u$.
A continuity argument shows that the forward iterates of the ranges (moving towards $\Kinv^u$) and the backward iterates of the kernels (moving towards $\Kinv^s$) meet half way at a sequence of parameters $t_k\to t$, thus creating null words
where  $L_1(\uA(t_k))=-\infty$.

\textbf{Case 3}: Suppose that $\uA_{\rm inv}(t)$ is not projectively uniformly hyperbolic.  We claim
that there is a sequence of parameters $t_n\to t$ in $\Lambda$ such that $\uA_{\rm inv}(t_n)$
admits   elliptic matrices with  irrational rotation numbers. Then both $\Wscr^+$ and $\Wscr^-$ are equal to $\Proj$  and by
the argument of Case 1 each $t_n$ can be approximated by parameters corresponding to null  words where $L_1=-\infty$ (see also Example \ref{IrratRot}).

We need two lemmas before we continue, the first of which is a simple adaptation of Yoccoz~\cite[Lemma 2]{Ycz04}.

\begin{lemma}
	\label{yoccoz lemma}
Given  a family of cocycles $\uA:I\to \SL_2(\R)^k$  that is  positively winding,
where $I\subset \R$ is an open interval, there exists a constant $c>0$ such that if for some $\delta>0$ and $t_0\in I$,
every matrix $A^n_t(\omega)$ is hyperbolic for   $\omega\in X$, $n\in \N$ and $\abs{t-t_0}<\delta$,  $t\in I$  then the cocycle $\uA(t_0)$ is uniformly hyperbolic with
$\norm{A^n_{t_0}(\omega)}\geq e^{c\, \delta\, n}$,
for all $n\geq 0$ and $\omega\in X$.
\end{lemma}

\begin{proof}
 Strictly speaking,~\cite[Lemma 2]{Ycz04} applies only to families as in Remark~\ref{rmk pos winding examples}. For the sake of completeness we sketch an adaptation of Yoccoz's argument.
 The real projective space is the equator of the complex projective space (the Riemann sphere)  $\Cp^1$ and splits it in two hemi-spheres $\HP^-$ and $\HP^+$, each of which can be identified as a hyperbolic plane.

  Now, given $1\leq i\leq k$  and $\hat v\in \Pp^1$, when we complexify the curve
 $\R\ni t\mapsto \hat A_i(t)\, \hat v\in \Pp^1$,
 the map
 $\C\ni t\mapsto \hat A_i(t)\, \hat v\in \Cp^1$,
 defined on a small strip $\mathscr{S}$ around $I$, takes the interval $I$ to the equator of $\Cp^1$ and because of the positive winding property it maps the two semi-strips
 $\mathscr{S}^\pm:= \{t\in \mathscr{S} \colon \mathrm{Im}(t)\lessgtr 0\}$ away from the equator. This behavior is uniform in $\hat v\in \Pp^1$.
 Hence, if $t\in \mathscr{S}^+$, resp. $t\in \mathscr{S}^-$, then the projective map $\hat A_i(t)\colon \Cp^1\to \Cp^1$ contracts $\HP^+$, resp. $\HP^-$,
 by a factor of order $\exp\left( -c\, |\mathrm{Im}(t)| \right)$, where $c$ is the winding speed. Notice that the alluded uniformity ensures that
 $\hat A_i(t)(\partial \HP^+)=\hat A_i(t)(\Pp^1)\Subset \HP^+$.

 The rest of the proof follows~\cite[Lemma 2]{Ycz04}. 
\end{proof}

\begin{lemma}
	\label{trace lemma}
	Let $A:I\subset \R\to \SL_2(\R)$ be a smooth positively winding family and	consider the function $f(t):=\tr(A(t))$.
	Then given $t_0\in I$:
	\begin{enumerate}
		\item If $|f(t_0)|<2$  or $|f(t_0)|=2$ but $A(t_0)\neq \pm I$  then  $f'(t_0)\neq 0$;
		\item If $A(t_0)= I$ and $f'(t_0)=0$  then $f''(t_0)<0$ and $t=t_0$ is a strict local maximum of $f(t)$;
		\item If $A(t_0)= -I$ and $f'(t_0)=0$  then $f''(t_0)>0$ and $t=t_0$ is a strict local minimum of $f(t)$.	
	\end{enumerate}
\end{lemma}

\begin{proof}	Because $\det(A(t))=1$ for all $t\in I$, we have  
	\begin{equation}
		\label{d det A}
		\tr( A(t)^{-1} \dot A(t)) =\frac{d}{dt}\det(A(t))=0 .
	\end{equation}

	\noindent
	(1).\; The first alternative ($|f(t_0)|<2$) holds because of~\cite[Proposition 3.21]{BCDFK}.
	We divide the second alternative ($|f'(t_0)|=2$ but $A(t_0)\neq \pm I$) in two sub-cases according to the sign of $f(t_0)$:
	
	\blob\;  $f(t_0)=2$ but $A(t_0)\neq I$. Assume for simplicity that this parabolic matrix takes the normal form
	$A(t_0)=\begin{bmatrix}
		1 & 1 \\ 0 & 1
	\end{bmatrix}$.
	If $f'(t_0)=0$ then
	$\dot A(t_0)=\begin{bmatrix}
		\kappa & \alpha \\ \beta & -\kappa 
	\end{bmatrix}$ for some constants
	$\alpha, \beta, \kappa$.
	Hence
	$$ A(t_0)^{-1} \dot A(t_0)= \begin{bmatrix}
		1 & -1 \\ 0 & 1
	\end{bmatrix}\,\begin{bmatrix}
		\kappa & \alpha \\ \beta & -\kappa 
	\end{bmatrix} = \begin{bmatrix}
		\kappa -\beta & \alpha+\kappa  \\ \beta & -\kappa 
	\end{bmatrix}
	$$
	which, because  $\tr(A(t_0)^{-1} \dot A(t_0))=0$, implies that $\beta=0$.
	Finally since the family is positively winding, the symmetric matrix
	$$ E^\sharp = J\, A(t_0)^{-1} \dot A(t_0) =
	\begin{bmatrix} 0 & 1 \\ -1 & 0
	\end{bmatrix} \,  \begin{bmatrix}
		\kappa  & \alpha+\kappa  \\ 0 & -\kappa 
	\end{bmatrix} =
	\begin{bmatrix}
		0 & -\kappa\\
		-\kappa  & -\alpha-\kappa   
	\end{bmatrix}$$
	is  positive definite by~\cite[Proposition 6.5]{BCDFK},  but on the other hand it has negative determinant $-\kappa^2$. This contradiction proves that $f'(t_0)\neq 0$.

	\blob\; \;  $f(t_0)=-2$ but $A(t_0)\neq -I$. We  assume  the parabolic matrix takes the normal form
	$A(t_0)=\begin{bmatrix}
		-1 & 1 \\ 0 & -1
	\end{bmatrix}$.
	If $f'(t_0)=0$ then again
	$\dot A(t_0)=\begin{bmatrix}
		\kappa & \alpha \\ \beta & -\kappa 
	\end{bmatrix}$ for some constants
	$\alpha, \beta, \kappa$.
	Hence
	$$ A(t_0)^{-1} \dot A(t_0)= \begin{bmatrix}
		-1 & -1 \\ 0 & -1
	\end{bmatrix}\,\begin{bmatrix}
		\kappa & \alpha \\ \beta & -\kappa 
	\end{bmatrix} = \begin{bmatrix}
		-\kappa -\beta & -\alpha+\kappa  \\ -\beta &  \kappa 
	\end{bmatrix}
	$$
	which, because $\tr(A(t_0)^{-1} \dot A(t_0))=0$ implies that $\beta=0$.
	Again,  by~\cite[Proposition 6.5]{BCDFK} the symmetric matrix
	$$ E^\sharp = J\, A(t_0)^{-1} \dot A(t_0) =
	\begin{bmatrix} 0 & 1 \\ -1 & 0
	\end{bmatrix} \,  \begin{bmatrix}
		-\kappa  & -\alpha+\kappa  \\ 0 &  \kappa 
	\end{bmatrix} =
	\begin{bmatrix}
		0 & \kappa\\
		\kappa  & \alpha-\kappa   
	\end{bmatrix}$$
	should be positive definite but it 
	has negative determinant $-\kappa^2$.
	This contradiction proves that $f'(t_0)\neq 0$.
	
	\bigskip
	
	\noindent
	(2) and (3).\; 
	Differentiating~\eqref{d det A} 
	$$ \tr\left( A(t)^{-1}\, \ddot A(t) -A(t)^{-1} \dot A(t)\, A(t)^{-1} \dot A(t)  \right)=0 , $$
	which implies that 
	$$ \tr\left( A(t)^{-1}\, \ddot A(t) \right) = \tr\left(   ( A(t)^{-1} \dot A(t)) ^2 \right) . $$

	We claim that the right-hand-side above is negative.
	In fact by~\eqref{d det A}  we have  $A(t)^{-1} \dot A(t)=\begin{bmatrix}
		\kappa & \alpha \\ \beta & -\kappa
	\end{bmatrix}$ and by~\cite[Proposition 6.5]{BCDFK} 
	the positive winding assumption implies that the symmetric matrix
	$$  E^\sharp = J\,A(t)^{-1} \dot A(t) 
	=\begin{bmatrix}
		\beta & -\kappa   \\ -\kappa & -\alpha
	\end{bmatrix} $$
	is  positive definite. Hence $-\alpha\, \beta -\kappa^2 >0$, or equivalently $\alpha\, \beta +\kappa^2 <0$.

	Finally, since
	$$  ( A(t)^{-1} \dot A(t)) ^2 
	=\begin{bmatrix}
		\alpha\,\beta + \kappa^2 & 0  \\ 0  & \alpha\,\beta + \kappa^2
	\end{bmatrix} $$
	its trace is negative, which proves the claim. In particular we have
	$$  \tr\left( A(t_0)^{-1}\, \ddot A(t_0) \right) =2\, (\alpha\, \beta+\kappa^2) <0 . $$
	If $A(t_0)=\pm I$ then $\pm f''(t_0)=\tr (A(t_0)^{-1}\, \ddot A(t_0))<0$ and the conclusions (2) and (3) hold.
\end{proof}
 
To prove the claim, normalizing the cocycle if necessary we may assume that $\uA_{\rm inv}$ takes values in $\SL_2(\R)$. Assume by contradiction
that for some open interval $J$ around $t$,
every iterate of the cocycle $\uA_{\rm inv}(t')$ is hyperbolic for all $t'\in J$. Then by  Lemma~\ref{yoccoz lemma}, the cocycle
$\uA_{\rm inv}(t)$ would be uniformly hyperbolic,
which contradicts the assumption.
Hence $\uA_{\rm inv}(t')$ has non hyperbolic periodic words for a sequence of parameters $t'=t_n$ converging to $t$. By Lemma~\ref{trace lemma}, with arbitrary small displacements of these parameters we obtain elliptic periodic orbits with irrational rotation numbers.

This concludes the proof the claim and hence of Theorem~\ref{ManeBochiMat2}.
\end{proof}

\begin{corollary} \label{cor1ManeBochi}
Given a rich analytic family $\uA \colon \Lambda \to \Matdm^k$, if $L_1(\uA(t_0)) > -\infty$, then either $L_1(\uA(t))$ is analytic around $t_0$ or else this parameter is a discontinuity point of $t\mapsto L_1(\uA(t))$.
\end{corollary}

\begin{proof}
By Theorem \ref{ManeBochiMat2}, either $\uA(t_0)$ is projectively uniformly hyperbolic or there exists a sequence $t_n \to t$ such that $L_1(\uA(t_n))=-\infty$. In the first case by a theorem of D. Ruelle \cite{Rue79a}, $L_1(\uA(t))$ is analytic in a neighborhood of $t_0$. In the second case, since we assume $L_1(\uA(t_0))>-\infty$ and there is a sequence $t_n \to t_0$ such that $L_1(\uA(t_n))=-\infty$, the Lyapunov exponent is discontinuous at $t_0$.
\end{proof}

\begin{definition}\label{sh:def}
A function $u\colon\Omega\to[-\infty,\infty)$ is called subharmonic in the domain $\Omega \subset \C$ if  $u$ is upper semi-continuous and for every $z \in \Omega$,
$$u (z) \le  \int_{0}^{1} u (z + r e^{2 \pi i \theta}) d \theta\,,$$
for some $r_0 (z) > 0$ and for all $r \le r_0 (z)$.
\end{definition}


Basic examples of subharmonic functions are $\log \sabs{z - z_0}$ or more generally, $\log \sabs{f (z)}$ for some analytic function $f (z)$ or $\int \log \sabs{z - \zeta} d \mu (\zeta)$ for some positive measure with compact support in $\C$.

The maximum of a finite collection of subharmonic functions is subharmonic, while the supremum of a collection (not necessarily finite) of subharmonic functions is subharmonic provided it is upper semi-continuous. 
In particular this implies that  if $A \colon \Omega \to \Matdm$ is a matrix valued analytic function on some open set $\Omega\subset \C$, then
$$u (t) := \log \norm{A(t)} = \sup_{\norm{v}, \norm{w} \le 1} \, \log \abs{\avg{A (t) \, v, w}}$$
is subharmonic in $\Omega$. 

Moreover, the infimum of a decreasing sequence of subharmonic functions is subharmonic. This shows that if $\uA \colon \Omega \to \Matdm^k$ is analytic, then the map $t \mapsto L_1(\uA(t))$ is subharmonic on $\Omega$.

\begin{lemma}\label{absl} Let $\Lambda \subset \R$ be a compact interval, let $\Omega \subset \C$ be an open, complex strip that contains $\Lambda$ and let
$u \colon \Omega \to [-\infty,\infty)$ be a subharmonic function such that   for some constant $C_0<\infty$ we have $u(z)\le C_0$ for all $z\in \Omega$ and $u(z_0)\ge - C_0$ for some $z_0\in \Omega$. 
 Then for all $N \in \N$,
$$\Leb \, \{t\in\Lambda \colon u(t)<-N\} \le  C \, e^{-N^\gamma}  ,$$ 
where $\gamma>0$ and $C$ is a finite constant depending on $C_0$,  $\Lambda$, $z_0$ and $\Omega$. 

In particular, $\Leb \, \{t\in\Lambda \colon u(t) = - \infty\} = 0 $.
\end{lemma}

\begin{proof} The statement is essentially the one-dimensional version of~\cite[Lemma 3.1]{Coposim}. In~\cite{Coposim} the set $\Omega$ is an annulus and $\Lambda$ is the torus, which are transformed into our setting via the complex logarithmic function.  
The proof of this result is a consequence of a quantitative version of the  Riesz representation theorem for subharmonic functions and of Cartan's estimate for logarithmic potentials, see~\cite[Lemma 2.2 and Lemma 2.4]{GS-fine}. 
\end{proof}

\begin{corollary} \label{cor2ManeBochi}
Given a rich analytic family $\uA \colon \Lambda \to \Matdm^k$, if there exists some $t_0 \in \Lambda$ such that $L_1(\uA(t_0))>-\infty$, then $L_1(\uA(t)) > -\infty$ for Lebesgue almost every $t \in \Lambda$. Therefore $L_1(\uA(t))$ satisfies an almost everywhere dichotomy: it is either analytic or discontinuous. 
\end{corollary}

\begin{proof}
 It is enough to prove the corollary when $\dim \Lambda=1$ and the family is positively winding.
Let $\Omega$ be a thin complex strip around $\Lambda$.

By continuity of the Lyapunov exponent,
if  $\Lambda$ is compact then $u(z):=L_1(\uA(z))$  admits a global  upper bound over $\Omega$ and if for some $t_0\in \Lambda$, $L_1(\uA(t_0)) >-\infty$, then Lemma~\ref{absl} is applicable, so that the set $\{ t \in \Lambda \colon L_1(\uA(t)) = -\infty \}$ has zero Lebesgue measure. The conclusion follows from Corollary~\ref{cor1ManeBochi}.
\end{proof}

\begin{remark}
In particular, the previous corollary holds for rich analytic families for which there exist $t_0 \in \Lambda$ such that $\uA(t_0)$ is projectively uniformly hyperbolic.
\end{remark}

We say that a set is \textit{residual} if it is a countable intersection of open and dense sets. The next corollary, is an adaptation of~\cite[Corollary 9.6]{Viana2014} and it shows that the set of continuity points of the Lyapunov exponent is a residual subset of $L^\infty(X,\Matdm)$.

Let $\mathcal{UH}$ denote the set of parameters $t \in \Lambda$ such that $\uA(t)$ is projectively uniformly hyperbolic.

\begin{corollary} \label{cor3ManeBochi}
The set $\{ t \in \Lambda \colon L_1(\uA(t)) = -\infty \}$ is a residual subset of $\Lambda \setminus \mathcal{UH}$.
\end{corollary}


\begin{proof}
Given $a\in\R$ let $\mathscr{L}(a):=\left\{t\in\Lambda \colon L_1(\uA(t))<a \right\}$. Since the top Lyapunov exponent is upper semi-continuous, $\mathscr{L}(a)$ is open in $\Lambda$. Consider a sequence $a_n\to -\infty$. The set in the statement coincides with $\cap_n   \mathscr{L}(a_n)$. By Theorem \ref{ManeBochiMat2} this intersection is dense in $\Lambda\setminus \mathcal{UH}$, hence so is each of the open sets $\mathscr{L}(a_n)$. Thus $\cap_n   \mathscr{L}(a_n)$ is a residual subset of $\Lambda \setminus \mathcal{UH}$.
\end{proof}

\begin{remark}
Since the map $A \mapsto L_1(A)$ is also upper semi-continuous for bounded and continuous cocycles with respect to the $C^0$ norm, it is also true that the set of cocycles $A \in C^0(X,\Matdm)$ which are either projectively uniformly hyperbolic, or else satisfy $L_1(A)= -\infty$, is a residual subset of $C^0(X,\Matdm)$.
\end{remark}


The set of constant rank $1$ cocycles
	$$ \Rscr_1 :=\left\{ \uA\in \Matdm^k \colon \rank(A_i)=1\; \forall\; 1\leq i\leq k\, \right\}$$
	is an analytic sub-manifold of $\Matdm^k$ with co-dimension $k$.

\begin{theorem} \label{ManeBochiRank1}
Given $\uA\in \Rscr_1$, either $d(\Kscr(\uA),\Rscr(\uA))>0$ and $\uA$ is projectively uniformly hyperbolic or $d(\Kscr(\uA),\Rscr(\uA))=0$, $\uA$ admits a null word  and $L_1(\uA)=-\infty$.
\end{theorem}

\begin{proof}
If $d(\Kscr(\uA),\Rscr(\uA))>0$,
there are two disjoint open subsets of $\Ascr\times \Proj$:
one forward invariant containing the ranges and another  
backward invariant containing the kernels. The first is 
an invariant multi-cone. Therefore, $\uA$ is projectively uniformly hyperbolic. 
If $d(\Kscr(\uA),\Rscr(\uA))=0$,  there exists a null word and  $L_1(\uA)=-\infty$.
\end{proof}

\begin{corollary}
\label{coro4ManeBochi}
For Lebesgue almost every $\uA\in\Rscr_1$, $L_1(\uA) > -\infty$. 
Moreover, the Lyapunov exponent $L_1:\Rscr_1\to [-\infty,+\infty[$ is continuous.
 and analytic on $\Rscr_1\setminus\{\uA\colon L_1(\uA)=-\infty\}$.
\end{corollary}

\begin{proof}
By Theorem \ref{ManeBochiRank1}, either $\uA$ is projectively uniformly hyperbolic or $L_1(\uA)=-\infty$. By Ruelle \cite[Theorem 3.1]{Rue79a}, if $\uA$ is projectively uniformly hyperbolic then the Lyapunov exponent is an analytic function. 
The first  Lyapunov exponent is continuous at points with $L_1(\uA)=-\infty$ (since it is an upper semi-continuous function). 
The set
\begin{align*}
	\left\{ \uA\in \Rscr_1  \colon L_1(\uA)=-\infty \right\} &= \left\{ \uA \in \Rscr_1 \colon \Kscr(\uA)\cap \Rscr(\uA)\neq \emptyset \right\}\\
	&=\bigcup_{i,j=1}^k \left\{ \uA\in \Rscr_1 \colon k(A_i)=r(A_j) \right\}\in \Rscr_1
\end{align*} 
is an algebraic sub-variety of $\Rscr_1$ with positive co-dimension in $\Rscr_1$. Hence it has zero Lebesgue measure in $\Rscr_1$.
\end{proof}

Finally we prove the results in the introduction.

\begin{proof}[Proof of Theorem~\ref{thm: BM type dichotomy}]
Consider the set $\Mscr$ of cocycles 
$\uA\in \Matdm^k$ with at least one singular component.
For each $I\subset \{1, \ldots, k\}$, define
\begin{align*}
\Mscr_I := & \left\{ \uA\in \Matdm^k\colon  A_i\neq 0, \, \det(A_i)=0\; \forall i\in I , \right. \\
&\left. \qquad \qquad \qquad \qquad    \det(A_j)\neq 0\; \forall j\notin I  \right\} .
\end{align*}
Each $\Mscr_I$ is the zero level set of the function  $f_I(\uA):= (\det(A_i))_{i\in I}$, whose Jacobian has maximal rank $m=\abs{I}$  along $\Mscr_I$.  Notice that
the row $i$ of the Jacobian of $f_I$ at $\uA$ is $(0,0, \ldots, \mathrm{adj}(A_i), \ldots, 0)$, where the zeros stand for the zero matrix in $\Matdm$ and the only non-zero entry, $\mathrm{adj}(A_i)$, is the cofactor matrix of $A_i$ and occurs at position $i$. Therefore
$\Mscr_I$ is an analytic manifold.

Since $\Mscr\setminus \cup_{I\subset \{1, \ldots, k\}} \Mscr_I$ consists
of $\uA\in \Matdm^k$ such that $A_i=0$ for some $i=1, \ldots, k$,   which are cocycles with $L_1(\uA)=-\infty$, it is enough to prove the dichotomy within each analytic manifold $\Mscr_I$. The conclusion follows applying Theorem~\ref{ManeBochiMat2}
to local analytic  parametrizations of $\Mscr_I$ whose ranges cover this manifold. These are rich families in view of Remark~\ref{rmk pos winding examples}, because the sets $\Mscr_I$ are rotation invariant, i.e.,
$\uA\in \Mscr_I$ implies that $R_t\, \uA\in \Mscr_I$.
\end{proof}

\begin{proof}[Proof of Corollary~\ref{coro: BM type dichotomy}]
Let $\Mscr$ be  the set of cocycles 
$\uA\in \Matdm^k$ with at least one singular and one invertible components and such that $L_1(\uA)>-\infty$. 
Consider also the analytic manifolds $\Mscr_I$ introduced in the previous proof.	
If $\uA\in \Mscr$ then $\uA\in \Mscr_I$ for some
$I\subset \{1, \ldots, k\}$ with $0<|I|<k$, because 
$\Mscr\setminus \cup_{0<|I|<k} \Mscr_I$ consists of cocycles with $L_1(\uA)=-\infty$.
The conclusion follows by applying Corollary~\ref{cor2ManeBochi}
to local analytic  parametrizations of $\Mscr_I$ whose ranges cover this manifold.
\end{proof}

\begin{corollary} \label{cor2ManeBochi}
Lebesgue almost every cocycle $\uA \in \Matdm^k$ satisfies $L_1(\uA)>-\infty$. In particular, this implies that almost every cocycle satisfies the regularity dichotomy: its Lyapunov exponent is either analytic or discontinuous. 
\end{corollary}

\begin{proof}
Given $\uA \in \Matdm^k$ with at least one invertible and one singular component, consider the $1$-parameter family $\uA(t)$ such that for every $i \in \Ainv$, $A_i(t) = A_iR_t$, where $R_t$ is a rotation and for $i \in \Asing$, $A_i(t)=A_i$. Consider $t \in [0,2\pi]$ and note that this family gives a foliation of $\Matdm^k$ by closed curves.

Let $\Omega$ be a thin, open complex strip around $[0,2\pi]$ and extend the analytic map $t\mapsto\uA(t)$ to a holomorphic map  $\uA \colon \Omega \to \Mat_2(\C)^k$.  Then $t \mapsto L_1(\uA(t))$ is subharmonic on $\Omega$. Moreover, on a slightly smaller compact set $\Omega'\subset \Omega$ containing $[0, 2\pi]$, since $\uA (t)$ is continuous hence bounded, there is $C < \infty$ such that $L_1(\uA(t)) \leq C$ for all $t \in \Omega'$.

Next we show that there is a point $z_0 \in \Omega'$ with $L_1 (\uA (z_0)) > - \infty$. In fact we show more (than we need), namely that for every $z$ with ${\rm Im} (z) \neq 0$, the cocycle $\uA (z)$ is projectively uniformly hyperbolic (so in particular $L_1 (\uA (z)) > - \infty$).
Without loss of generality we may assume that
$\uA_{\rm inv}(t)$ takes values in $\SL_2(\R)$.
The argument sketched in the proof of Lemma~\ref{yoccoz lemma} shows that if ${\rm Im} (t)>0$ then 
$\uA_{\rm inv}(t)$ is uniformly hyperbolic 
with $\Kinv^u\Subset \HP^+$ and  $\Kinv^s\Subset \HP^-$.
On the other hand,  the singular matrices of the cocycle take the whole projective space, except the kernel, to a point in the equator of the Riemann sphere, and because they are constant in $t$, their ranges and kernels stay away from $\Kinv^u$ and $\Kinv^s$. Therefore, $\uA(t)$ is always uniformly hyperbolic when $\mathrm{Im}(t)>0$.

We conclude that the subharmonic function $\uA (t)$, defined in a neighborhood of the interval $[0, 2\pi]$, is bounded from above and it takes a finite value at a point close to the interval. Lemma~\ref{absl} is then applicable and it implies that $L_1 (\uA (t)) > - \infty$ for Lebesgue almost every point $t\in[0, 2\pi]$.

Thus for every curve $\{ \uA(t) \colon  t\in [0,2\pi]\} \subset \Matdm^k$ and almost every parameter $t$, the cocycle $\uA(t)$ satisfies $L_1(\uA(t))>-\infty$. By Fubini we then conclude that almost every cocycle $\uA \in \Matdm^k$ satisfies $L_1(\uA)>-\infty$ and, in particular, almost every cocycle $\uA \in \Matdm^k$ satisfies the regularity dichotomy. 
\end{proof}

\section{Examples}\label{examples}
\begin{example}
\label{explo 1}
We provide now an example of a cocycle $\uA$ such that 
$\uA_{\rm inv}$ is uniformly  hyperbolic, 
$\uA$ is diagonalizable,
we have  $\mathcal{W}^+\cap\mathcal{W}^-=\emptyset$ but   $\uA$ is not uniformly hyperbolic.
Let $\underline{A}=(A_0,B_0)$ be the cocycle taking  values
$A_0=\begin{bmatrix}
2 & 0\\
0 & 2^{-1}
\end{bmatrix}$ 
and 
$B_0=\begin{bmatrix}
0 & 0\\
0 & 1
\end{bmatrix}$
with probability measure $\mu=\frac{1}{2}\,\delta_{A_0}+\frac{1}{2}\, \delta_{B_0}$. Notice that $L_1(\uA)=-\log 2$.
This cocycle is not projectively uniformly hyperbolic. Consider the sets 
 $$\Sigma_n=\{(A,B): \ A \text{ is hyperbolic}, \, \rank(B)=1, \ A^n\, \hat r(B)=\hat k(B)\, \},$$ 
 and the family of cocycles $\uA_t:=\left(A_0,B_t:=\begin{bmatrix}
-t^2 & t\\
-t & 1
\end{bmatrix}\right)$,  in the para\-me\-ter $t\in\R$, such that $\uA=\uA_0$. The cocycle $\uA=(A_0,B_0)$ is an accumulation point of the cocycles $\uA_{t_n}$, where $t_n=2^{-n}$, and so it is also an accumulation point of the hypersurfaces $\Sigma_n$ as $\uA_{t_n}\in\Sigma_n$. It follows from
Theorem~\ref{thm2: branch charact of UH}  that in a neighborhood $\mathcal{U}$ of  $\underline{A}$, the projectively uniformly hyperbolic cocycles are precisely those in  $\mathcal{U}\setminus(\overline{\bigcup_{n=1}^{\infty}\Sigma_n})$. Then $L_1(\uA_t)$ has an isolated singularity, with $L_1(\uA_{t_n})=-\infty$, at each $t=t_n$.
In particular $\uA=\uA_0$  can not be projectively uniformly hyperbolic.
\end{example}

\begin{example} \label{IrratRot}
This example appears in the introduction of~\cite{AEV} as well as in~\cite[Section 3]{DF24}.
Consider the family of cocycles $\uA_t:=(A,B_t)$, where $A=\begin{bmatrix}
1 & 0\\
0 & 0
\end{bmatrix}$ and $B_t=\begin{bmatrix}
\cos t & -\sin t\\
\sin t & \cos t
\end{bmatrix}$ are chosen with probabilities $p=(\frac{1}{2}, \frac{1}{2})$.
By~\cite[Introduction]{AEV} or~\cite[Proposition 3.1]{DF24}:

 \begin{enumerate}
\item\; 	$\displaystyle L_1(\uA_t)=\sum_{j=0}^{\infty}\frac{1}{2^{j+1}}\log\abs{\cos(jt)}.$
\item\; If $t	\in\pi\,\mathbb{Q}$ there exists $n\in\mathbb{N}$ such that $AB^n_tA=0$ and so $L_1(\underline{A_t})=-\infty$. 
\item\; The set $\{t\in\mathbb{R}:L_1(\underline{A_t})>-\infty\}$ has full Lebesgue measure. 
\item\; If $t\in \pi\, (\mathbb{R}\setminus\mathbb{Q})$ then $\mathcal{W^+}\cap\mathcal{W^-}=\Pp^1$.
\end{enumerate}
Hence, by $(2)$ and $(3)$ we conclude that the function $t\mapsto L_1(\uA_t)$ is discontinuous for almost every $t\in\mathbb{R}$.
\end{example}

\begin{example}
Fix  $p=(p_1, p_2, p_3)$ a probability vector with positive entries, take $\lambda >\sqrt{2}$ and let $A_\lambda=\begin{bmatrix}
\lambda & 0\\
0 & \lambda^{-1}
\end{bmatrix}$ and $B_\lambda=\begin{bmatrix}
\lambda & 0\\
\lambda-\lambda^{-1} & \lambda^{-1}
\end{bmatrix}$. 
 Consider then the space of cocycles 
$$  \mathcal{X}_\lambda =\{  \uA=(A_\lambda, B_\lambda, C)\, \colon\,   \rank(C)=1,\,  \mathrm{range}(C) \notin  \Kinv^s \, \}  .
$$
For $\uA\in \mathcal{X}_\lambda$,
  \, $\uA_{\rm inv}=(A_\lambda,B_\lambda)$ is uniformly hyperbolic,\,  $\Kinv^s=\ker(C)$ \, and \,   $\Kinv^u$ is a Cantor set. 
 Consider also the natural map $\pi\colon \{1,2\}^{\N}\rightarrow \Kinv^u$ defined by
 $\pi(\omega):=\lim_{n\to \infty} A_{\omega_{0}}\, A_{\omega_1}\, \cdots \, A_{\omega_{n-1}}\, \hat e_1$, where $e_1=(1,0)$, $A_1=A_\lambda$ and $A_2=B_\lambda$.
 For $\omega\in \{1,2\}^\Z$   define the leaves
 $$\mathcal{X}_{\lambda,\omega}:=\{(A_\lambda, B_\lambda, C)\in \mathcal{X}_\lambda\, : \, \pi(\omega) = \ker(C)\,\}$$
which foliate  the set
$$\Fscr_\lambda :=\bigcup_{\omega\in\{1,2\}^{\N}}\mathcal{X}_{\lambda,\omega} .  $$  
This is  a Cantor like foliation of $\Fscr_\lambda$  into algebraic hypersurfaces. 
\begin{lemma}
The set $\Fscr_\lambda$ consists  of the cocycles $\uA\in \mathcal{X}_{\lambda}$ that are not projectively uniformly hyperbolic.
\end{lemma}

\begin{proof}
	Assume $\uA\in \mathcal{X}_{\lambda}\setminus\Fscr_\lambda$.
	Since  $\mathrm{range}(C)\notin \Kinv^s$, we have that 	$\Wscr^+=\Kinv^u$ and $\Kscr(\uA)\cap \Wscr^+=\emptyset$.
	Also $\ker(C)\notin\Kinv^u$, which implies that $\Wscr^-=\Kinv^s$ and $\Rscr(\uA)\cap \Wscr^-=\emptyset$.

	By Theorem~\ref{thm2: branch charact of UH},
	since $\uA$ is not diagonalizable and $\uA_{\rm inv}$ is 
	projectively uniformly hyperbolic, we conclude that 
	$\Wscr^-\cap \Wscr^+=\emptyset$ and $\uA$ is projectively uniformly hyperbolic.

	Conversely if $\uA\in \Fscr_\lambda$ then \,
	$\ker(C)\in \Kinv^u$, which implies that
	$\Kscr(\uA)\cap \Wscr^+\neq \emptyset$.
	Hence by Theorem~\ref{thm2: branch charact of UH}, the cocycle $\uA$ is not projectively uniformly hyperbolic.	
\end{proof}

For $\omega\in\{1,2\}^n$ define the leaves
$$\mathcal{Y}_{\lambda,\omega}:=\{\uA= (A_\lambda,B_\lambda, C) : \ \rank(C)=1, \ A^n(\omega)\, \mathrm{range}(C)=\ker(C)\},$$ 
consisting of cocycles $\uA\in \mathcal{X}_\lambda$ which admit  the null word $C\, A^n(\omega)\, C$ and satisfy $L_1(\uA)=-\infty$.
Since each leaf $\mathcal{X}_{\lambda,\omega}$ can be approximated by the leaves $\mathcal{Y}_{\lambda,[\omega]_n}$,
where $[\omega]_n=(\omega_0, \ldots, \omega_n)$,   every cocycle $\uA\in \mathcal{X}_{\lambda,\omega}$ with $L_1(\uA)>-\infty$ is a discontinuity point of   $L_1$.

Given an analytic curve  of cocycles $I\ni t\mapsto \uA_t\in \mathcal{X}_\lambda$ that crosses transversely all leaves
$\Fscr_{\lambda, \omega}$, with $\omega\in \{1,2\}^\Z$, the set of parameters
$\Sigma_\lambda:=\{t\in I\colon \uA_t\in \Fscr_\lambda\}$
is a Cantor   diffeomorphic to $\Kinv^u$ where either $L_1=-\infty$ or else $L_1$ is discontinuous. The complement $I\setminus\Sigma_\lambda$ is an open set where $L_1$ is analytic.
\end{example}

Consider a Bernoulli shift $\sigma\colon \Ascr^\Z\to \Ascr^\Z$ endowed with a Bernoulli measure $\mu=p^\Z$ in a finite alphabet $\Ascr=\{1, \ldots, k\}$
and let $V:\Ascr^\Z\to \R$ be a locally constant function,
$V(\omega)=V(\omega_0)$. This determines a family of $1$-dimensional random Schr\"odinger operators \,	
$H_{V, \omega}:\ell^2(\Z)\to \ell^2(\Z)$ defined by
$$ (H_{V, \omega})\psi_n:= -\psi_{n+1}-\psi_{n-1} + V(\sigma^n\omega)\, \psi_n \quad \psi=(\psi_n)_{n\in \Z}\in \ell^2(\Z) .$$
By  Pastur Theorem, the spectra $\sigma(H_{V, \omega})$ of these bounded linear operators
are the same for $\nu^\Z$ almost every $\omega\in \Ascr^Z$, see~\cite[Theorem 3.6]{Dam-survey}. The properties of   $\sigma(H_V)$  are strongly connected with those of the Lyapunov spectrum of the  family of 	
Schr\"odinger cocycles
$$\uA (t)= \left( \begin{bmatrix}
	V(2)-t & -1 \\ 1 & 0 
\end{bmatrix}, \, \ldots \, , \begin{bmatrix}
	V(k)-t & -1 \\ 1 & 0 
\end{bmatrix} \right)$$
where the parameter $t\in \R$ represents the energy of the quantum physical model.

Random Schr\"odinger cocycles are always non compact and strongly irreducible, see~\cite[Theorem 4.2]{Dam-survey} and by H. Furstenberg criterion~\cite[Theorem 8.6]{Fur63}
their Lyapunov exponents  are simple,
$L_1=-L_2>0$. Thus by E. Le Page~\cite[Th\'{e}or\`{e}me 1]{LP89} its Lyapunov exponent is  H\"older continuous in  $t$, which in turn implies that the Integrated Density  of States (IDS), a core  measurement in the spectral theory of Schr\"odinger operators, is also H\"older continuous in the energy $t$. These quantities, the Lyapunov exponent and the IDS,  are related via the Thouless formula,~\cite[Theorem 3.16]{Dam-survey}, which implies that they  share the same regularity, see~\cite[Lemma 10.3]{GS01}.

\begin{example}\label{example CS}
We provide a formal model of a Schr\"odinger  cocycle associated to a $1$-dimensional  Schr\"odinger operator with random  potential taking two values, one finite and another  equal to $+\infty$, as in W.  Craig and B. Simon~\cite[Example 3]{CraigSimon}. 

Let $\uA_t:=(A_t(1), A_t(2)) \in \Mat_2(\R)^2$ be 
given by $A_t(1) \equiv \begin{bmatrix}
	1 & 0 \\ 0 & 0 
\end{bmatrix}$ and
$A_t(2)=\begin{bmatrix}
	a-t & -1 \\ 1 & 0 
\end{bmatrix}$  where $a\in \R$ is fixed.  Then $\uA_t$ is the 
asymptotic limit as $\la\to\infty$ of the rescaled random Schr\"odinger cocycle 
$$\uA_{t,\la}=(\uA_{t,\la}(1), \uA_{t,\la}(2))\quad \text{ 
where }\quad  \uA_{t,\la}(j)=\la_j^{-1}\, \begin{bmatrix}
	v_j-t &-1\\  1 & 0 	
\end{bmatrix}  $$
with $\la_1=\la$ and $\la_2=1$.
The random Schr\"odinger operator $H_\la$ with potential $V$ taking   value  
$V=\la$ with probability $p$  and   value $V=a$  with probability $1-p$, respectively,
has spectrum  
$$\sigma( H_\la)=[a-2,a+2] \cup [\la-2,\la+2],$$ see~\cite[Theorem 4.1]{Dam-survey},
and as a function of the energy $t$, its Lyapunov exponent is H\"older continuous on $\sigma(H_\la)$ and real analytic on $\R\setminus \sigma(H_\la)$.

B.  Halperin in~\cite{Halperin67}, see also~\cite[Theorem A.3.1]{ST85}, proved that for this specific model  the exponent in the H\"older  regularity  of both the IDS and the Lyapunov exponent decays  to $0$ as $\la\to +\infty$.  
This relates well with the fact that in~\cite[Example 3]{CraigSimon} the IDS is discontinuous
and  the Lyapunov exponent of the limiting cocycle $\uA_t$ is also almost everywhere discontinuous in $[a-2,a+2]$.
Indeed, for $t\in [a-2,a+2]$, the matrix $\uA_t(2)=\begin{bmatrix}
	a-t & -1 \\ 1 & 0
\end{bmatrix}$ is elliptic 
and hence the equation
$$ \uA_t(2)^n\, \begin{bmatrix}
	1 \\ 0
\end{bmatrix} =\begin{bmatrix}
 0 \\ 1
\end{bmatrix}$$
has infinitely many dense solutions in $[a-2,a+2]$, which correspond to null words of $\uA_t$.
On the other hand for  $t\notin [a-2,a+2]$,
the matrix $\uA_t(2)$ is hyperbolic and there exist no null words. In particular $\uA_t$ is projectively uniformly hyperbolic in this case.

It is not  be difficult to see that for every $t\in \R$,
$$ \lim_{\la\to +\infty} L_1(\uA_{\la, t})=L_1(\uA_t)$$
(but we omit the proof) which implies that the Schr\"odinger cocycle with potential $V$ (with values $a$ and $\lambda$) 
has Lyapunov exponent
$$ L_1(S_{\la, t})\sim p\, \log \la + L_1(\uA_t) \text{ as } \la\to +\infty .$$

\end{example}

\medskip

\subsection*{Acknowledgments}

P.D. was partially supported by FCT - Funda\c{c}\~{a}o para a Ci\^{e}ncia e a Tecnologia, through
the projects   UIDB/04561/2020 and  PTDC/MAT-PUR/29126/2017.
M.D. was supported by a CNPq doctoral fellowship. This study was financed in part by the Coordena\c{c}\~{a}o de Aperfei\c{c}oamento de Pessoal de N\'{i}vel Superior - Brasil (CAPES) - Finance Code 001.
T.G. was supported by FCT - Funda\c{c}\~{a}o para a Ci\^{e}ncia e a Tecnologia, through the projects UI/BD/152275/2021 and by CEMS.UL Ref. UIDP/04561/2020, DOI 10.54499/UIDP/04561/ \\2020.
S.K. was supported by supported by the CNPq research grant 313777/2020-9 and  by the Coordena\c{c}\~ao de Aperfei\c{c}oamento de Pessoal de N\'ivel Superior - Brasil (CAPES) - Finance Code 001.

\bigskip

\bibliographystyle{amsplain}
\bibliography{bib}

\end{document}